%% file: GBCforSemiabelian.tex
\RequirePackage{currfile} 

\documentclass[11pt, english]{smfart}

\input{autres_packages.tex}
\input{macros.tex}

\title[]{Geometric Bogomolov conjecture for semiabelian varieties}

\author{Wenbin Luo}
\email{luowenbin\_math@outlook.com}
\address{School of Mathematical Sciences, Shanghai Key Laboratory of PMMP, East China Normal University, 500 Dongchuan Road, Shanghai, 200241}
\author{Jiawei Yu}
\email{2201110030@pku.edu.cn}
\address{School of Mathematical Sciences, Peking University, Road Yiheyuan No.5, 100871, Beijing, China}

\date{\today} 

\begin{document}
\begin{abstract}
    We establish the geometric Bogomolov conjecture for semiabelian varieties over function fields. We show a closed subvariety contains Zariski dense sets of small points, if and only if, after modulo its stabilizer, it is a torsion translate of a constant variety. A new phenomenon is that a special subvariety may not have a Zariski dense set of points of height $0$.
    
\end{abstract}
\maketitle
\section{Introduction}
\subsection{Background and main results}
Let $K/k$ be a finitely generated field extension such that $k$ is algebraically closed in $K$.
We fix a \textit{model} of $K/k$, that is, a projective normal $k$-variety $\mathfrak B$ such that $k(\mathfrak B)=K$. We further equip $\mathfrak B$ with an ample line bundle $\mathcal H$ on $\mathfrak B$.
Let $A$ be an abelian variety over $K$, on which we fix an ample line bundle $N$ that is \textit{symmetric} i.e., $[-1]^*N\simeq N$. Then the Néron-Tate height $\widehat h_{N}:A(\ovl K)\rightarrow \R$ is a quadratic form on $A(\ovl K)$. We say a subvariety $X\subset A_{\ovl K}$ \textit{contains Zariski dense sets of small points} if, for any $\epsilon>0$, $X(\epsilon):=\{x\in X(\ovl K)\mid \widehat h_{N}(x)\leq \epsilon\}$ is Zariski dense in $X$. The celebrated geometric Bogomolov conjecture {\bf (GBC)} states that $X$ contains Zariski dense sets of small points if and only if $X$ is \textit{special} in the sense of \cite{Yamaki2013gbc}. 

This is an analogue of the Bogomolov conjecture over number fields, which was proved by Ullmo \cite{Ullmo1998Bogomolov} for curves embedded in their Jacobians, and by Zhang \cite{Zhang1998Bogomolov} in the general case. In \cite{Moriwaki2000Height}, Moriwaki generalized the Bogomolov conjecture to the case over finitely generated field over $\Q$ by introducing an arithmetic height function.

Surprisingly, the geometric Bogomolov conjecture is significantly more intricate than its arithmetic equivalent. Its proof has a rich history involves contributions from multiple mathematicians. In \cite{Gubler2007GBC}, Gubler proved a crucial result showing that {\bf(GBC)} holds for somewhere totally degenerate abelian variety by applying his tropical equidistribution theorem. His method and its generalization \cite{Gubler2010nonA} was then largely exploited by Yamaki in \cite{Yamaki2016strict}. In addition, Yamaki \cite{Yamaki2017nondense} proved the case that $X$ is of dimension $1$ or codimension $1$ in $A$, using his crucial reduction theorem in \cite{Yamaki2018trace} showing that it suffices to prove the case that $A$ is nowhere degenerate and has trivial $K/k$-trace. These ultimately led to the proof of the geometric Bogomolov conjecture by Xie--Yuan \cite{YX2022GBC}. 

If $k$ is of characteristic $0$, Cinkir obtained an effective result for curves embedded in their Jacobians \cite{Cinkir2011Bogomolov}. In \cite{GH2019Bogomolov}, Gao--Habegger proved the case that the transcendental degree of $K/k$ is $1$ by using their height inequality.  By constructing the Betti form via dynamical systems, Cantat--Gao--Habegger--Xie \cite{CGHX2021GBC} provided a proof of {\bf (GBC)} in the characteristic $0$ case.

A semiabelian variety is an extension of an abelian variety by a torus. The Bogomolov conjecture for semiabelian varieties over number fields is proved by Chambert-Loir \cite{cham2000pet} in the almost-split case, by David--Phillipon \cite{DP2000Semiabel} in the general case (via a reduction to the Manin--Mumford conjeture), and by Kühne \cite{Kuhne2024semibogo} using his equidistribution theorem. In this article, we present a generalization of {\bf (GBC)} for semiabelian varieties. 

Throughout this article,
let $G$ be a semiabelian variety over $K$ given by an exact sequence of group varieties
$$0\rightarrow \mathbb{G}_{m}^t\rightarrow G\xrightarrow{\pi} A\rightarrow0$$
where $A$ is an abelian variety defined over $K$. Such an exact sequence corresponds to an element $\underline \eta:=(\eta_1,\dots,\eta_t)\in (A^\vee(\ovl K))^t$ by the Weil-Barsotti theorem.
After replacing $K$ by a finite extension, we may further assume that $\underline \eta\in (A^\vee(K))^t$

Each $\eta_i$ can be viewed as a translation-invariant line bundle $Q^{(i)}$ on $A$. We take the compactification 
$$\overline G:=\mbb P(\mc O_A\oplus (Q^{(1)})^\vee)\times_A\cdots\times_A\mbb P(\mc O_A\oplus (Q^{(t)})^\vee)$$ as in \cite{Kuhne2024semibogo}. Let $M$ be the line bundle associated with the boundary divisor $\overline G\setminus G$ as in \S \ref{subsec_height_semi}. Then the multiplication by $n$ map on $G$ extends to $[n]:\ovl G\rightarrow \ovl G$ such that $[n]^*M\simeq nM$. This gives a canonical height $\widehat h_{M}:\ovl G(\ovl K)\rightarrow \R_{\geq 0}$ via Tate's limiting process
$$\widehat h_{M}(x):=\lim_{n\rightarrow +\infty}\frac{1}{n}h_{M}([n]x)$$
for $x\in G(\ovl K)$, where $h_M$ is a Weil height function associated with $M$. 

Let $L=M+\ovl\pi^*N$, where $\ovl\pi: \ovl G\rightarrow A$ extends $\pi$. We have the canonical height $\widehat h_{L}:=\widehat h_{M}+\widehat h_{N}\circ \ovl\pi:\ovl G(\ovl K)\rightarrow\R$ associated with $L$.
Note that this compactification is different from \cite{cham2000pet}, but the heights are comparable.

For any closed subvariety $X\subset G_{\ovl K}$, we denote by $\mr{Stab}(X)$ the closed subset $\{y\in G_{\ovl K}\mid y+X=X\}$ equipped with the reduced scheme structure, and by $\mr{Stab}_0(X)$ the neutral component of $\mr{Stab}(X)$ which is a semiabelian variety due to \cite[Corollary 5.4.6 (1)]{Brion2017str}. In this article, we prove the following geometric Bogomolov type result.
\begin{theo}[Geometric Bogomolov Conjecture]\label{theo_GBC}
    Let $X\subset G_{\ovl K}$ be a closed subvariety. Then the following are equivalent:
    \begin{enumerate}
        \item[\textnormal{(a)}] For any $\epsilon>0$, $\{x\in X(\ovl K)\mid \widehat h_{L}(x)<\epsilon\}$ is Zariski dense in $X$.
        \item[\textnormal{(b)}] There exist 
        \begin{itemize}
            \item a torsion point $x$ in the semiabelian variety $\widetilde G:=G_{\ovl K}/\mr{Stab}_0(X)$.
            \item a semiabelian variety $G_0$ over $k$ with a homomorphism $h:G_0\ot_k \ovl K\rightarrow \widetilde G$ of finite kernel.
            \item a closed $k$-subvariety $X_0\subset G_0$,
        \end{itemize}  
        such that $X/\mr{Stab}_0(X)=h(X_0\ot_{k}\ovl K)+x.$
    \end{enumerate}
\end{theo}
A closed subvariety satisfying condition (b) is called \emph{special}. Note that being special is independent of the choice of $L$ and the completion of $G$.

\subsection{New phenomenon and applications}
In the aforementioned Bogomolov type results, we can see that if a closed subvariety contains Zariski dense sets of small points, then points of canonical height $0$ are Zariski dense in $X$. This may not always be the case for semiabelian varieties over function fields. We give an explicit example.
\begin{exem}\label{exam_nondense}
    Let $k=\C$ and $E$ be an elliptic curve over $\C$. Set $\mathfrak B=E$. Then $K$ is the function field of $E$. The Poincaré line bundle $\mc P$ on $E\times E$ gives a semiabelian variety $$0\rightarrow \mbb G_m\rightarrow G\xrightarrow{\pi} E_{\ovl K}\rightarrow 0.$$
    Taking a non-torsion point $y\in E(k)$, the canonical height on the fiber $\ovl\pi^{-1}(y\ot_k \ovl K)\simeq \mathbb P^{1}_{\ovl K}$ is induced by the model 
    $$(f:\mbb P(\mc O_{E}\oplus \mc P_y)\rightarrow E,\mc M),$$ where $\mc M=2O(1)-f^*\mc P_y$. Here $O(1)$ is the relative tautological line bundle of $f$. By the non-triviality of $\mc P_y$, this height only vanishes at zero and infinity on $\mbb P^1(\ovl K)$, while $\pi^{-1}(y\ot_k\ovl K)$ is clearly special with the stabilizer $\mbb G_m$.
\end{exem}

Assume that $k=\ovl {\mbb F}_p$. A direct application of our result is the following: if the maximal nowhere degenerate abelian subvariety of $A$ is constant (see \cite{Yamaki2018trace} for the definitions), then Theorem \ref{theo_GBC} implies the Manin--Mumford conjecture in the positive characteristic case. The reason we need to add this assumption is that {\bf (GBC)} for abelian varieties is proved in this case without reducing to the Manin--Mumford conjecture in loc. cit..

We note that the Manin--Mumford conjecture in positive characteristic was originally proved by Hrushovski using model theory \cite{Hrushovski2001MMC}. Inspired by Hrushovski's approach, Pink-Rossler \cite{PinkRossler2002MMC,PinkRossler2004MMC} gives another proof using only algebraic geometry. 

\subsection{Outline of the proof}
We sketch the proof of $\textnormal{(b)}\Rightarrow\textnormal{(a)}$. We first prove the case that $G$ is \emph{quasi-split}, namely, $G$ is isogenous to $G_0\times_k A_1$ where $G_0$ is a semiabelian variety over $k$, and $A_1$ is an abelian variety over $K$ with trivial $\ovl K/k$-trace (see \S\ref{subsec_trace} for the definition). This was done by generalizing Yamaki's theory of relative height \cite{Yamaki2018trace}. In general, we can reduce to this case by using what we call the \emph{relative Faltings--Zhang} maps:
\[\begin{tikzcd}
     &G^n_{/A}:=\underbrace{G\times_A\cdots\times_A G}_{n\text{ times}} \arrow[r,"\alpha_n"]\arrow[rd,"\beta_n"] & \mbb G_m^{t(n-1)}\times A\\
     & & \mbb G_m^{t(n-1)}\times G\arrow[u,"\mr{id}\times \pi"].
    \end{tikzcd}
\]
where $\alpha_n(x_1,\dots, x_n):=(x_1-x_2,\dots,x_{n-1}-x_n,\pi(x_1))$ and $\beta_n(x_1,\dots, x_n):=(x_1-x_2,\dots,x_{n-1}-x_n,x_n).$

Let $S=\pi(X)$ and $\eta$ be the generic point of $S$. Using a similar argument as in \cite[Lemma 3.1]{Zhang1998Bogomolov}, we can show that if $X$ is of trivial stabilizer and $X_\eta$ is geometrically irreducible, then the fibered product $X^n_{/S}$ is generically finite to $\alpha_n(X^n_{/S})$ for $n\gg0$. In particular, $\beta_n(X^n_{/S})\rightarrow\alpha_n(X^n_{/S})$ is generically finite, since $\beta_n$ is an isomorphism.

Note that $\alpha_n(X^n_{/S})\subset \mbb G_m^{t(n-1)}\times A$ contains Zariski dense sets of small points, by the quasi-split case, $\alpha_n(X^n_{/S})$ is special. 
For simplicity, we consider the example that $\alpha_n(X^n_{/S})=A$.

Observe that $\beta_n(X^n_{/A})\subset \mbb G_m^{t(n-1)}\times G$ also contains Zariski dense sets of small points. The idea is as follows: $\beta_n(X^n_{/A})$ is "almost a subgroup" in $\mbb G_m^{t(n-1)}\times G$, which is also "almost a section" of $A$. This should happen only when $G$ is quasi-split.

Actually, combining with Lemma \ref{lemm_aux} and Theorem \ref{theo_constant_torsion_test}, we proved that if \begin{enumerate}
    \item $X$ contains Zariski dense sets of small points,
    \item $\mr{Stab}(X)$ is trivial, and 
    \item $\pi(X)$ generates $A$, i.e. $\Z\pi(X)$ is Zariski dense in $A.$
\end{enumerate}then $G$ is quasi-split. 
Therefore, for a general $X$ containing Zariski dense sets of small points, after taking the quotient by the stabilizer of $X$, and shrinking $A$ if necessary, we can reduce to the quasi-split case.

As for the proof of Lemma \ref{lemm_aux}, we consider dynamical systems of form: $$[n^2]\times[n]:\mbb G_m^{t'}\times G\rightarrow \mbb G_m^{t'}\times G,$$ together with a modified fundamental inequality in Arakelov geometry (see Theorem \ref{theo_fund_ineq}). 

In addition, our approach offers a new proof of the Bogomolov conjecture for semiabelian varieties over $\ovl \Q$, by considering the number field analogue of Theorem \ref{theo_constant_torsion_test}. 
\subsection{Structure of the article}
In \S\ref{sec_GCH_SEMI}, we begin by recalling the theory of adelic line bundles, adelic divisors, arithmetic intersection theory, and prior results on abelian varieties. In particular, we provide the construction of canonical line bundles on semiabelian varieties over a function field.

In \S\ref{sec_GBC_SEMI}, we formulate the geometric Bogomolov conjecture. We prove that {\bf(GBC)} holds for $G$ if and only if it holds for any $G'$ in its isogeny class.

In \S\ref{sec_YMK_RVS}, we generalize Yamaki's theory of relative heights to prove {\bf (GBC)} in the quasi-split case, i.e., $G$ is isogenous to $G_0\times_k A_1$ where $G_0$ is a semiabelian variety over $k$, and $A_1$ is an abelian variety over $K$ with trivial $\ovl K/k$-trace (see \S\ref{subsec_trace} for the definition).

In \S\ref{sec_REL_FZ}, we begin by proving the equivalence between containing Zariski dense sets of small points and vanishing of certain arithmetic intersections numbers (see Lemma \ref{lemm_aux}). As an application, we show that a special subvariety contains Zariski dense set of small points. We then use the trick of relative Faltings--Zhang map to reduce the general case to the quasi-split case, relying on Lemma \ref{lemm_aux} and Theorem \ref{theo_constant_torsion_test}.

In \S\ref{sec_Appendix A} Appendix A, we offer a proof for our modified fundamental inequality (see Theorem \ref{theo_fund_ineq}). The absolute minimum in the original fundamental inequality proved in \cite[Theorem 5.2]{zhangposvar} and \cite[Proposition 4.3]{Gubler2007GBC} is replaced by a new invariant called \textit{sectional minimum} (see Definition \ref{defi_ess_abs}). 

In \S\ref{sec_Appendix_B} Appendix B, we outline a new proof for Bogomolov conjecture for semiabelian varieties over number fields, which is purely Arakelov-theoretical. We essentially reduce to the almost-split case which was proved by Chambert-Loir \cite{cham2000pet}.

\section*{Notation and conventions}
In this article, a \emph{variety} refers to an integral separated scheme of finite type over a field. This also applied to the notion of a \emph{subvariety}.

Let $k$ be an algebraically closed field.
The field $K$ under consideration is a function field $k(\mathfrak B)$ of a normal projective variety $\mathfrak B$ over $k$. Set $$\mathfrak B^{(1)}:=\{\nu\in \mathfrak B\mid \nu\text{ is of codimension }1\text{ in }\mathfrak B\},$$
which corresponds to the set of \emph{places} of $K$. Each $\nu\in \mathfrak B^{(1)}$ gives an absolute value $\lvert\cdot\rvert_\nu:K\rightarrow \R_{\geq 0}, \lambda\mapsto \exp(-\mr{ord}_\nu(\lambda)),$ where $\mr{ord}_\nu(\cdot)$ is the order of the discrete valuation ring $\mc O_{\mathfrak B,\nu}.$ We denote by $K_\nu$ the completion of $K$ with respect to $\lvert\cdot\rvert_\nu$, and by $K_\nu^\circ$ the valuation ring $$K_\nu^\circ:=\{\lambda\in K_\nu\mid \lvert\lambda\rvert_\nu\leq 1\},$$ We fix an ample line bundle $\mc H$ on $\mathfrak B$. We denote by $\deg_{\mc H}(\nu)$ the degree
$\mc H^{\dim {\mathfrak B}-1}\cdot[\ovl {\{\nu\}}]$ for $\nu\in \mathfrak B^{(1)}$.

Let $X$ be a projective $K$-variety. Let $\mc S$ be an integral scheme with function field $K$.
We say a projective integral $\mc S$-scheme $f:\mc X\rightarrow \mc S$ is an $\mc S$\textit{-model} of $X$, if $\mc X_K$ is isomorphic to $X$, where $\mc X_K$ is the generic fiber of $\mc X$. Let $L$ be a line bundle on $X$. We say $(f:\mc X\rightarrow \mc S,\mc L)$ is an $\mc S$-model of $(X,L)$ if $\mc X_K=X$ and $\mc L_K=L$. If $\mc S=\mr{Spec}(R)$ for some integral domain $R$, we say $(\mc X,\mc L)$ is a $R$-model of $(X, L)$. If $\mc S=\mathfrak B$, we simply say $(\mc X,\mc L)$ is a model of $(X, L)$ if there is no ambiguity.

Let $X_1$ and $X_2$ be projective $K$-varieties. Let $L_i$ $(i=1,2)$ be line bundles on $X_i$ respectively. We denote by $L_1\boxtimes L_2$ the line bundle $p_1^*L_1+p_2^*L_2$ on $X_1\times X_2$, where $p_i$ $(i=1,2)$ is the $i$-th projection $X_1\times X_2\rightarrow X_i.$
\section{Geometric canonical heights on semiabelian varieties}\label{sec_GCH_SEMI}
\subsection{Non-archimedean metrics and measures}
Let $X$ be a projective variety over $K$, and fix a place $\nu\in\mathfrak B^{(1)}$. Let $K_\nu$ be the completion of $K$ with respect to the absolute value $\lvert\cdot\rvert.$
We denote by $X^\an_\nu$ the analytification of $X_{K_\nu}$ with respect to $\lvert\cdot\rvert_\nu$ in the sense of Berkovich \cite{Berkovich}. $X_\nu^\an$ consists of pairs $x:=(p,\lvert\cdot\rvert_x)$, where $p$ is a scheme point of $X_{K_\nu}$, and $\lvert\cdot\rvert_x$ is an absolute value on the residue field $\kappa(p)$ extending $\lvert\cdot\rvert_\nu$. $X_\nu^\an$ is equipped with \emph{the Berkovich topology} \cite[\S 1]{Berkovich}. 

For each $x\in X_\nu^\an$, let $\widehat\kappa(x)$ denote the completion of $\kappa(p)$ with respect to $\lvert\cdot\rvert_x$.
Let $L$ be a line bundle over $X$. A $\nu$-adic metric $\phi_\nu$ on $L$ is a collection $\{\lVert\cdot\rVert_{\phi_\nu}(x)\}_{x\in X^{\an}_\nu}$, where each $\lVert\cdot\rVert_{\phi_\nu}(x)$ is a norm on $L_x\ot_{\kappa(p)}\widehat\kappa(x)$. 

We say $\phi_\nu$ is \textit{continuous} if for any open subset $U\subset X$ and $s\in \Gamma(L,U)$, the function $$(x\in U^\an_{\nu})\mapsto \lVert s\rVert_{\phi_\nu}(x)$$ is continuous with respect to the Berkovich topology. 

If there exists a $K^\circ_\nu$-model $(\mc X,\mc L)$ of $(X,\ell L)$ for some positive integer $\ell$, then the induced \textit{model metric} on $L$ is defined as in \cite[Lemma 7.4]{gubler2006local}. By \cite[Theorem 7.12]{gubler2006local}, any continuous metric is a uniform limit of model metrics. 

The model metric is said to be \textit{semipositive} if $\mc L$ is relatively nef. A continuous metric $\phi_\nu$ is said to be \emph{semipositive} if it is a uniform limit of semipositive model metrics. A pair $(L,\phi_\nu)$ is then referred to a \textit{semipositive metrized line bundle}. 

Semipositive metrized line bundles generate an abelian group, whose elements are called \emph{admissible metrized line bundle}.

The following is due to Chambert-Loir \cite{Chambert2006Measures}:
\begin{prop}
    Let $(L_1,\phi_{1,\nu}),\dots,(L_{\dim X},\phi_{\dim X,\nu})$ be admissible metrized line bundles on $X$. There exists a unique way to associate a regular Borel measure $c_1(L_1,\phi_{1,\nu})\cdots c_1(L_{\dim X},\phi_{\dim X,\nu})$ on $X^\an_\nu$ which is multi-linear and symmetric in $(L_1,\phi_{1,\nu}),\dots,(L_{\dim X},\phi_{\dim X,\nu})$.
\end{prop}

\subsection{Adelic line bundles over function fields}
Consider a pair $\ovl L:=(L,\{\phi_\nu\}_{\nu\in \mathfrak B^{(1)}})$ where each $\phi_\nu$ is a continuous $\nu$-adic metric on $L$. We say that $\ovl L$ is an \textit{adelic line bundle} if
there exists a sequence $\{(\mc X_n,\mc L_n,\ell_n)\}_{n\in\N_+}$, 
such that for each $n$, $(\mc X_n,\mc L_n)$ is a $\mathfrak B$-model of $(X,\ell_n L)$ for $\ell_n\in\N_+,$
and the following conditions hold:
\begin{enumerate}
    \item For each $\nu\in \mathfrak B^{(1)}$, the model metric induced by $(\mc X_{K_\nu^\circ},\mc L_n\ot{K_\nu^\circ},\ell_n)$ converges uniformly to $\phi_\nu,$ as $n\rightarrow+\infty.$
    \item There exists an open subset $\mathfrak U\subset\mathfrak B$, such that for all $\nu\in \mathfrak U\cap \mathfrak B^{(1)}$, $\phi_\nu$ is the model metric induced by $(\mc X_{K_\nu^\circ},\mc L_n\ot{K_\nu^\circ},\ell_n)$ for every $n.$
\end{enumerate}

We say an adelic line bundle $\ovl L=(L,\{\phi_\nu\}_{\nu\in \mathfrak B^{(1)}})$ is \textit{semipositive} if all $\phi_\nu$ are semipositve. An adelic line bundle $\ovl L$ is called \textit{integrable} if it can be written as $\ovl L=\ovl L_1-\ovl L_2$ for semipositive adelic line bundles $\ovl L_1$ and $\ovl L_2$.

We denote by $\widehat{\mr{Pic}}_{\mr{int}}(X)$ the group of integrable adelic line bundles.
For each $\nu\in \mathfrak B^{(1)}$, we denote by $\ovl L_{\nu}$ the admissible metrized line bundle $(L,\phi_\nu).$
\begin{prop}
    We can define a multi-linear and symmetric form \begin{align*}
    \widehat{\mr{Pic}}_{\mr{int}}(X)^{\dim X+1}&\rightarrow \R,\\
    (\ovl L_1,\ovl L_2\dots,\ovl L_{\dim X+1})&\mapsto \ovl L_1\cdot \ovl L_2\cdots\ovl L_{\dim X+1},
\end{align*}
such that the following conditions are satisfied:
\begin{enumerate}
    \item[\textnormal{(1)}] If $\ovl L_i$ are induced by line bundles $\mc L_i$ on a $\mathfrak B$-model $f:\mc X\rightarrow \mathfrak B$ of $X$, then $$\ovl L_1\cdot \ovl L_2\cdots\ovl L_{\dim X+1}=\mc L_1\cdots\mc L_{\dim X+1}\cdot f^*{\mc H}^{\dim \mathfrak B-1}.$$
    \item[\textnormal{(2)}] Let $s$ be a rational section of $L_{\dim X+1}.$ Assume that the cycle associated with $\mr{div}(s)$ is $\sum_{i} \lambda_i Z_i$ where $\lambda_i\in\Z$ and $Z_i$ are prime divisors. Then \begin{align*}
        &\ovl L_1\cdot \ovl L_2\cdots\ovl L_{\dim X+1}=\sum \lambda_i\ovl L_1|_{Z_i}\cdots\ovl L_{\dim X}|_{Z_i}\\&\kern 5em-\sum_{\nu\in \mathfrak B^{(1)}}\deg_{\mc H}(\nu)\int_{X_\nu^\an}\log\lVert s\rVert_{\phi_{n,\nu}}c_1(\ovl L_{1,\nu})\cdots c_1(\ovl L_{\dim X,\nu}).
    \end{align*}
\end{enumerate}
\end{prop}

Let $\ovl K$ be an algebraic closure of $K$. Let $Y\subset X_{\ovl K}$ be a closed subvariety. Let $\ovl L_1,\cdots,\ovl L_{\dim Y+1}$ be adelic line bundles on $X$. We are going to define the intersection number $\ovl L_1\cdots\ovl L_{\dim Y+1}\cdot [Y]$. We may assume that $Y$ is defined over a a finite extension $K'$ of $K$. Let $\mathfrak B'$ be the normalization of $\mathfrak B$ in $K'$. We then obtain a natural map $\mathfrak B'^{(1)}\rightarrow \mathfrak B^{(1)}$. For any $\nu'\in \mathfrak B'^{(1)}$ lying over $\nu\in \mathfrak B^{(1)}$, there is a canonical continuous map $(X_{K'})^{\an}_{\nu'}\rightarrow X^{\nu}$ defined in the natural way. 
Replacing $\ovl L_i$ by their pull-backs to $X_{K'}$. The intersection number is defined as
$$\ovl L_1\cdots \ovl L_{\dim Y+1}\cdot [Y]:=\ovl L_{1}|_{Y}\cdots\ovl L_{\dim Y+1}|_Y,$$
which is independent of the choice of $K'$.
The \textit{height} of $Y$ with respect to an adelic line bundle $\ovl L$ is defined as $$h_{\ovl L}(Y):=\frac{\ovl L^{\dim Y+1}\cdot [Y]}{(\dim Y+1)\deg_{L}(Y)}.$$
We say $\ovl L$ is \emph{nef} if it is semipositive and for any closed subvariety $Y\subset X_{\ovl K}$, $h_{\ovl L}(Y)\geq 0.$
\begin{defi}\label{defi_ess_abs}
    Let $\ovl L$ be an adelic line bundle on $X$. The \textit{essential minimum} of $\ovl L$ on $X$ is defined as 
$$\mr{ess}_{\ovl L}(X):=\sup_{U\mathop{\subset}\limits_{open} X}\inf_{x\in U(\ovl K)}h_{\ovl L}(x).$$ 

    For $\nu\in \mathfrak B^{(1)}$, $s\in H^0(X,L)$, define that $$\lVert s\rVert_{\phi_\nu}^{\sup}:=\sup_{x\in X^\an_\nu}\lVert s\rVert_{\phi_\nu}(x).$$
    
    The \emph{Arakelov degree} of $s$ is defined as
    $$\widehat{\deg}_{\ovl L}(s):=-\sum_{\nu\in\mathfrak B^{(1)}}\deg_{\mc H}(\nu)\log\lVert s\rVert^{\sup}_{\phi_\nu},$$
    with the convention that $\widehat{\deg}_{\ovl L}(0):=-\infty.$

 We define the \emph{sectional minimum}
    $$\lambda_{\min}(\ovl L):=\sup\left\{t\in\R\middle | \begin{aligned}
        \forall \text{closed }Z\subsetneq X, \exists s\in H^0(X,nL),s.t. \widehat\deg_{n\ovl L}(s)\geq nt,\\\text{and } s\text{ does not vanish at generic points of }Z
    \end{aligned}\right\}.$$
\end{defi}

\begin{theo}[Fundamental inequality]\label{theo_fund_ineq}
    Let $\ovl L$ be a semipositive adelic line bundle on $X$. Then $$\mr{ess}_{\ovl L}(X)\geq h_{\ovl L}(X)\geq \frac{1}{\dim X+1}(\mr{ess}_{\ovl L}(X)+\dim X\cdot \lambda_{\min}(\ovl L)).$$
    In particular, if the geometric part of $\ovl L$ is ample on $X$, then $\lambda_{\min}(\ovl L)>-\infty.$
\end{theo}
\begin{proof}
    See \S\ref{sec_Appendix A} Appendix A.
\end{proof}
\subsection{Adelic divisors}
Let $D$ be a Cartier divisor on $X$.
A green function $g_\nu$ of $D$ on $X_\nu^\an$ is a continuous function on $(X\setminus \mr{Supp}(D))_\nu^\an$, such that for any local equation $f$ of $D$ on an open subset $U$ of $X$, the function $g_\nu+\log\lvert f\rvert_\nu$ extends to a continuous function on $U^\an_\nu$. Conversely, one can equip $\mc O_X(D)$ the metric $\lVert\cdot\rVert_{g_\nu}$ defined by $\lVert s_D\rVert_{g_\nu}:=\exp(-g_\nu)$, where $s_D$ is the rational section of $\mc O_X$ corresponding to $D$.

An \textit{adelic divisor} on $X$ is a pair $\ovl D=(D,\{g_\nu\}_{\nu\in \mathfrak B^{(1)}})$ where $D$ is a Cartier divisor on $X$, and each $g_\nu$ is Green function of $D$ on $X^{\an}_\nu$, such that 
$$\mc O_X(\ovl D):=(\mc O_X(D), \{\lVert\cdot\rVert_{g_\nu}\}_{\nu\in\mathfrak B^{(1)}})$$
is an adelic line bundle.
We say $\ovl D$ is \textit{effective} if $D$ is effective and $g_\nu\geq 0$ for every $\nu$.
Note that an adelic divisor can be defined by a sequence $\{(\mc X_n,\mc D_n,\ell_n)\}$ such that each $(\mc X_n,\mc D_n)$ is a model of $(X, \ell_n D)$, and $\{(\mc X_n,\mc O_{\mc X_n}(\mc D_n),\ell_n)\}$ defines an adelic line bundle.

\subsection{Canonical line bundles on abelian varieties}\label{subsec_abel_can}
Let $N$ be an ample line bundle over $A$. We first assume that $N$ is symmetric, i.e., $[-1]^*N\simeq N$. The theorem of the cube gives $[m]^*N= m^2N$ for integers $m$. Tate's limiting argument yields a canonical adelic line bundle $\ovl N$ such that $[m]^*\ovl N=m^2\ovl N$ (see \cite[Theorem 9.5.7]{BG2006Heights}). Let $Q$ be a translation-invariant line bundle on $A$. After replacing $K$ by a finite extension, we may write $Q=T_x^*N-N$ for some $x\in A(K).$  Then the canonical line bundle associated with $Q$ is given by $\ovl Q=T_x^*\ovl N-\ovl N-x^*\ovl N$ as shown in \cite[Corollary 2.2]{cham2000pet}. 

\subsection{Heights on semiabelian varieties}\label{subsec_height_semi}
By the Weil-Barsotti formula \cite[\S III.18]{Oort1966group}, the exact sequence $0\rightarrow \mbb G_m^t\rightarrow G\rightarrow A\rightarrow 0$ corresponds to an element $(\eta_1,\cdots, \eta_t)\in (A^\vee(K))^t$.
We denote by $Q^{(i)}$ the translation-invariant line bundle on $A$ corresponding to $\eta_i$. Consider the projective compactification: $$\ovl G:=\mbb P(\mc O_A\oplus (Q^{(1)})^{\vee})\times_A\cdots\times_A\mbb P(\mc O_A\oplus (Q^{(t)})^\vee).$$
Note that the multiplication by $n$ map $[n]:G\rightarrow G$ can be extended to $\ovl G$. We continue to denote this by $[n]$, by abuse of notation.
Let $D^{(i)}_{[0]}$ (resp. $D^{(i)}_{[\infty]}$) be the divisor in $\ovl {G}$ associated with $\mc O_{A}\oplus (Q^{(i)})^\vee\twoheadrightarrow (Q^{(i)})^\vee$ (resp. $\mc O_{A}\oplus (Q^{(i)})^\vee\twoheadrightarrow \mc O_{A}$). Set $$M:=\mc O_{\ovl G}\left(\sum_{1\leq i\leq t}(D^{(i)}_{[0]}+D^{(i)}_{[\infty]})\right),$$ which satisfies $[n]^*M=nM.$
We are going to construct an adelic line bundle $\ovl M$ satisfying $[n]^*\ovl M=n\ovl M$ via the limiting process.

The following is well known to experts. Here we include the proof for the reader's convenience.
\begin{prop}\label{prop_nef_model}
    Let $Q$ be a translation-invariant line bundle on $A$.
    There exists a finite extension $K'/K$, a projective morphism $\mc A\rightarrow \mathfrak B'$ where $\mathfrak B'$ is the normalization of $\mathfrak B$ in $K'$, and a line bundle $\mc Q$ on $\mc A$ such that
    \begin{enumerate}
        \item[(i)] $(\mc A,\mc Q)$ is a $\mathfrak B'$-model of $(A,Q)$,
        \item[(ii)] $\mc Q_\nu$ is nef on $\mc A_\nu$ for each $\nu\in \mathfrak B'^{(1)}$.
    \end{enumerate}    
\end{prop}
\begin{proof}
    Take a projective $\mathfrak B$-model $\mc A_0$ of $A$. Then
    there exists an open subset $\mathfrak U\subset \mathfrak B$ such that $\mc A_0\times_{\mathfrak B}{\mathfrak U}\rightarrow \mathfrak U$ is an abelian scheme. 

    Since $Q$ is translation-invariant, there exists a finite extension $K'/K$, a geometrically reduced projective curve $C/K'$, $P_1,P_2\in C(K')$, and a divisor $E$ om $A\times C$ such that $Q$ is the line bundle corresponds to the divisor $p_{1,*}(E\cdot p_2^*([P_1]-[P_2])).$ Let $\phi:\mathfrak B'\rightarrow \mathfrak B$ be the normalization of $\mathfrak B$ in $K'$. Then there are finitely many $\nu_1,\dots,\nu_m\in \mathfrak B'^{(1)}\setminus \phi^{-1}(\mathfrak U_0)$. Let $\mc O_i:=\mc O_{\mathfrak B',\nu_i}$ for $i=1,\dots,m$, which are discrete valuation rings. 
    
    For each $i$, let $\mc C_i$ be a semistable $\mc O_i$-model of $C$. Then we can extend $P_1-P_2$ to a divisor $\mc D_i$ on $\mc C_i$ such that $\mc D_i\cdot V=0$ for any curve $V$ contained in the special fiber of $\mc C_i.$

    By \cite[Theorem 4.2]{Kunn1998model}, there exists a projective regular scheme $\mc P_i$ flat over $\mr{Spec}(\mc O_i)$, whose generic fiber is $A_{K'}.$ Let $\mc E$ be a divisor of $\mc P_i\times C_i$ extending $E$. Then the line bundle $\mc Q_i$ corresponds to $p_{1,*}(\mc E\cdot p_2^* \mc D_i)$ extends $Q$, and is nef on the special fiber of $\mc P_i$, by the same calculation in \cite[Lemma 8.2]{Kunnemann_1996}.

    After taking blowing-ups of $\mc A_0\times_{\mathfrak B}{\mathfrak B'}$, we can obtain a $\mathfrak B'$-model $\mc A$ of $A_{K'}$, admitting $\mc A\rightarrow \mc A_0\times_{\mathfrak B}{\mathfrak B'}$ and $\mc A\times_{\mathfrak B'}{\mc O_i}\rightarrow \mc P_i$ for $i=1,\dots, m$. 

    We take a symmetric ample line bundle $N$ on $A_{K'}$. We may assume that $Q=T_x^*N-N$ for $x\in A(K')$. Let $\mc N$ be the rigidified symmetric line bundle on $\mc A'_{\mathfrak U}$, extending $N$. Then the line bundle $\mc Q':=T_x^*\mc N-\mc N-x^*\mc N$ is algebraically equivalent to $0$ on each fiber of $\mathfrak U$. Gluing up $\mc Q_1,\dots, \mc Q_m$ and $\mc Q'$, we obtain the desired line bundle $\mc Q$ on $\mc A$.
\end{proof}

From now on, we replace $\mathfrak B$ by a finite covering $\mathfrak B'$ as in Proposition \ref{prop_nef_model}.
\begin{theo}\label{theo_eff_boud_div}
    There exist adelic divisors $\ovl D_{[0]}^{(i)}$(resp. $\ovl D_{[\infty]}^{(i)}$) on $\ovl G$ with underlying divisors $D_{[0]}^{(i)}$(resp. $D_{[\infty]}^{(i)}$) such that 
    \begin{enumerate}
        \item[\textnormal{(a)}] $\ovl D_{[0]}^{(i)}$ and $\ovl D_{[\infty]}^{(i)}$ are effective and semipositive.
        \item[\textnormal{(b)}] $[n]^*\ovl D_{[0]}^{(i)}=n\ovl D_{[0]}^{(i)}$ and $[n]^*\ovl D_{[\infty]}^{(i)}=n\ovl D_{[\infty]}^{(i)}$.
        \item[\textnormal{(c)}] $\ovl D_{[\infty]}^{(i)}-\ovl D_{[0]}^{(i)}$ is an adelic divisor associated with $\pi^* \ovl Q^{(i)}.$
    \end{enumerate}
\end{theo}
\begin{proof}
We briefly explain the limiting process. After taking a finite extension of $K$, we can
take model $\mc A$ of $A$, and models $\mc Q^{(i)}$ of $Q^{(i)}$ as in Proposition \ref{prop_nef_model}. Then $\mc {\ovl G}:=\mbb P(\mc O_{\mc A}\oplus \mc Q_1^\vee)\times_{\mathfrak B}\cdots\times_{\mathfrak B}\mbb P(\mc O_{\mc A}\oplus \mc Q_t^\vee)$ is a $\mathfrak B$-model of $\ovl G$. Let $\mc D_{[0]}^{(i)} \text{ (resp. }\mc D_{[\infty]}^{(i)})$ be the divisors corresponds to $\mc O_{\mc A}\oplus \mc Q^{(i)}\twoheadrightarrow \mc O_{\mc A}\text{ (resp. }\mc O_{\mc A}\oplus \mc Q^{(i)}\twoheadrightarrow \mc Q^{(i)}).$ Note that $\mc O_{\ovl {\mc G}}(\mc D_{[\infty]}^{(i)}-\mc D_{[0]}^{(i)})\simeq \pi^* \mc Q^{(i)}$ where $\pi$ denotes $\ovl{\mc G}\rightarrow \mc A$ by abuse of notation.

For each $m$, by abuse of notation, we denote by $[m]:\mc A_m\rightarrow \mc A$ the normalization of the composition $A\xrightarrow{[m]} A\rightarrow \mc A.$ Similarly, let $[m]:\ovl {\mc G}_m'\rightarrow \ovl {\mc G}$ is the normalization of $\ovl G\xrightarrow{[m]} \ovl G\rightarrow \ovl{\mc G}$. Set $\mc{\ovl G}_m:=\ovl {\mc G}_m'\times_{\mc A} \mc A_m$. Then the sequence $$\{(\mc{\ovl G}_m,[m]^*\mc D_{[0]}^{(i)},m)\}_{m\in \N_+}$$ defines an adelic divisor $\ovl D_{[0]}^{(i)}$ that satisfies $[n]^*\ovl D_{[0]}^{(i)}=n \ovl D_{[0]}^{(i)}$. The adelic divisor $\ovl D^i_{[\infty]}$ is defined by the sequence $\{(\mc{\ovl G}_m,[m]^*\mc D_{[\infty]}^{(i)},m)\}$, whence (b) is verified. Since $$\mc O_{\mc{\ovl G}_m}([m]^*(\mc D_{[\infty]}^{(i)}-\mc D_{[0]}^{(i)}))=[m]^*\pi^*\mc Q=\pi_m^*[m]^*\mc Q$$ where $\pi_m:\mc{\ovl G}_m\rightarrow\mc A_m$, (c) is satisfied.

By construction, (a) is due to the effectivity and relative nefness of $\mc D_{[\infty]}^{(i)}$ and $\mc D_{[0]}^{(i)}$.
\end{proof}

Let $$\ovl M:=\mc O_{\ovl G}\left(\sum_{1\leq i\leq t} \ovl D_{[0]}^{(i)}+\ovl D_{[\infty]}^{(i)}\right)$$ be the adelic line bundle on $\ovl G.$ 
Let $N$ be a symmetric ample line bundle on $A$. The limiting process gives a canonical line bundle $\ovl N$ on $A$ as aforementioned. Set 
$$L:=M+\ovl\pi^*N\text{, and }\ovl L:=\ovl M+\ovl\pi^*\ovl N$$
where $\ovl \pi:\ovl G\rightarrow A$ extends $\pi$.
Note that $L$ is ample on $\ovl G$.
For $x\in \ovl G(\ovl K)$, the \textit{canonical height} is defined as $\widehat h_L(x):=h_{\ovl L}(x)$.

\section{Geometric Bogomolov conjecture for semiabelian varieties}\label{sec_GBC_SEMI}
For the completeness of the article, in this section, we show that {\bf{(GBC)}} for a semiabelian variety holds if it holds for any semiabelian variety in its isogeny class. The reader may skip \S \ref{subsec_trace} and \S \ref{subsec_special}, after reproducing the proof of \cite[Lemma 3.3]{Ziegler2015MLp} in our context.
\subsection{Basics on semiabelian varieties}
We recall some fundamental results on semiabelian varieties.
Let $\mathbf{k}$ be an arbitrary field. In this article, $\mathbf k$ could be $k$, $K$ or $\ovl K.$ 
Let $0\rightarrow \mathbb G_{m}^{t}\rightarrow G\rightarrow A\rightarrow 0$ and $0\rightarrow \mathbb G_{m}^{t'}\rightarrow G'\rightarrow A'\rightarrow 0$ be semiabelian varieties over $\mathbf{k}$. By the Weil-Barsotti formula, these extensions correspond to $\underline \eta=(\eta_1,\cdots,\eta_{t})\in (A^\vee(\mathbf{k}))^{t}=\mr{Ext}^1_{\mathbf k}(A,\mbb G_m^{t})$ and $\underline \eta'=(\eta'_1,\cdots,\eta'_{t'})\in (A'^\vee(\mathbf{k}))^{t'}=\mr{Ext}^1_{\mathbf k}(A',\mbb G_m^{t'})$. 

Let $\phi:G\rightarrow G'$ be a group homomorphism. By \cite[Lemma 1]{Kuhne2020bounded}, there exist homomorphisms $\phi_{\mr{tor}}:\mbb G_m^t\rightarrow \mbb G_m^{t'}$ and $\phi_{\mr{ab}}:A\rightarrow A'$, such that the following diagram is commutative:
\[\begin{tikzcd}
     &0 \arrow[r] &\mathbb G_{m}^{t}\arrow[r]\arrow[d,"\phi_{\mr{tor}}"]& G\arrow[r,"\pi"]\arrow[d,"\phi"]& A\arrow[r]\arrow[d,"\phi_{\mr{ab}}"] &0\\
     &0 \arrow[r] &\mathbb G_{m}^{t'}\arrow[r]&G'\arrow[r,"\pi'"]& A'\arrow[r]&0.
    \end{tikzcd}
\]
Conversely, for every pair $(\phi_{\mr{tor}},\phi_{\mr{ab}})$ such that $\phi_{\mr{tor},*}\underline\eta=\phi_{\mr{ab}}^*\underline\eta'$, there exists a unique $\phi:G\rightarrow G'$ satisfying the above commutative diagram.

Note that there exists a compactification of homomorphism $\ovl \phi: G_{\ovl{\Gamma(\phi_{\mr{tor}})}}\rightarrow \ovl G'$ as constructed in \cite[Construction 7]{Kuhne2020bounded}. Here $\Gamma(\phi_{\mr{tor}})\subset \mbb G_m^t\times \mbb G_m^{t'}$ denotes the graph of $\phi_{\mr{tor}}$, and $\ovl{\Gamma(\phi_{\mr{tor}})}$ is its closure in $(\mbb P^1)^t\times (\mbb P^1)^{t'}.$  $G_{\ovl{\Gamma(\phi_{\mr{tor}})}}$ is the $\mbb G_m^t$-equivariant compactification of $G$. In particular, if $\phi_{\mr{tor}}$ is the multiplication by $n$ map for some positive integer $n$, then $G_{\ovl{\Gamma(\phi_{\mr{tor}})}}=\ovl G$. We refer the reader to loc. cit. for the precise definition of this compactification.

If $\mathbf k'/\mathbf k$ is a field extension, we denote by $\underline \eta\ot_{\mathbf k}\mathbf k'$ the base change of $\underline \eta$ via the natural map $\mr{Ext}_{\mathbf k'}^1(A_{\mathbf k'},\mbb G_m^t)\rightarrow \mr{Ext}_{\mathbf k}^1(A_{\mathbf k},\mbb G_m^t)$.
\subsection{Chow $\ovl K/k$-traces}\label{subsec_trace}
From now on, assume that $G,G',A,A'$ in previous subsections are defined over $K$.
We denote by $(A_{\ovl K/k},\mr{tr})$ the $\ovl K/k$-trace of $A_{\ovl K}$, that is, the final object of the category of pairs $(B,h)$, where:\begin{itemize}
    \item $B$ is an abelian variety over $k$,
    \item $h:B\ot_k\ovl K\rightarrow A_{\ovl K}$ is a morphism of abelian varieties.
\end{itemize} Note that $\mr{tr}$ has finite kernel. For more details, we refer the reader to \cite[\S 8.3]{Lang1959AV} or \cite[\S 6]{Conrad2006trace}. 

Similarly, for a semiabelian variety $G_{\ovl K}$, we can define the category of pairs $(G_0,h)$ where $G_0$ is a semiabelian variety over $k$, and $h$ is a morphism from $G_0\ot_k\ovl K$ to $G_{\ovl K}$. The existence of the final object is established in \cite{long2024trace}. In this article, we only use the following result, which is used to prove Proposition \ref{prop_special_isogeny}.

\begin{prop}\label{prop_trace_isogeny}
    Let $\phi:G_{\ovl K}\rightarrow G_{\ovl K}'$ be an isogeny.
    Let $H$ be a semiabelian variety over $k$, and $f:H\ot_k \ovl K\rightarrow G'_{\ovl K}$ be a morphism with finite kernel. Then there exist
    \begin{enumerate}
        \item[\textnormal{(i)}] semiabelian varieties $G_0$ and $G_0'$ over $k$, and an isogeny $\psi:G_0\rightarrow G_0',$
        \item[\textnormal{(ii)}] homomorphisms $h:G_0\ot_k \ovl K\rightarrow G_{\ovl K}$, and $h':G_0'\ot_k\ovl K\rightarrow G'_{\ovl K}$, each with finite kernel,
        \item[\textnormal{(iii)}] a homomorphism $g: H\rightarrow G_0'$.
    \end{enumerate} such that $f=h'\circ (g\ot_k\ovl K)$, and the following diagram is commutative:
    \[\begin{tikzcd}
     &G_0\ot \ovl K\arrow[r,"\psi\ot_k\ovl K"]\arrow[d,"h"] &G'_0\ot \ovl K\arrow[d,"h'"]\\
     & G \arrow[r,"\phi"] &G'.
    \end{tikzcd}
    \]
\end{prop}
\begin{proof}
    Assume that $H$ is defined by the exact sequence 
    \begin{equation*}\label{eq_exact_G'_0}
        0\rightarrow \mathbb G_m^{t'}\rightarrow H\rightarrow A_0'\rightarrow 0.
    \end{equation*}
    Let $(A_{\ovl K/k},\mr{tr})$ (resp. $(A'_{\ovl K/k},\mr{tr}')$) be the $\ovl K/k$-trace of $A$ (resp. $A'$). By the universal property of the trace, we may assume that $A'_0\subset A'_{\ovl K/k},$ that is, $$f_{\mr{ab}}=\mr{tr}'|_{A'_0\ot_k \ovl K}.$$

    Take $\underline\eta'\in \mr{Ext}^1_{\ovl K}(A'_{\ovl K},\mathbb G_m^t)$ and $\underline\eta_{H}\in \mr{Ext}^1_{k}(A_0',\mathbb G_m^{t'})$) which represent the extensions that define $G'_{\ovl K}$ and $H$ respectively. Then we have $$f_{\mr{ab}}^*\underline \eta'=f_{\mr{tor},*}(\underline\eta_H\ot \ovl K)=f_{\mr{tor},*}(\underline\eta_H)\ot \ovl K.$$ 
    Let $\underline \eta_0':=f_{\mr{tor},*}(\underline\eta_H)$ and $G_0'$ be the semiabelian variety over $k$ defined by $\underline \eta'_0.$
    There exists a homomorphism $h':G_0'\ot_k \ovl K\rightarrow G'_{\ovl K}$ induced by $h'_{\mr{tor}}:=\mr{id}_{\mathbb G_m^t}$ and $h'_{\mr{ab}}:=f_{\mr{ab}}.$ Therefore $f$ factors through $h'$.
    
    Since $\phi_{\mr{ab}}$ is an isogeny, there exists an isogeny $\mr{tr}(\phi_{\mr{ab}}):A_{\ovl K/k}\rightarrow A'_{\ovl K/k}$, such that $\mr{tr'}\circ \mr{tr}(\phi_{\mr{ab}})_{\ovl K}=\phi_{\mr{ab}}\circ \mr{tr}.$ Let $A_0$ be the neutral component of $\mr{tr}(\phi_{\mr{ab}})^{-1}(A_0')$. If we set $h_{\mr{ab}}:=\mr{tr}|_{A_0\ot_k \ovl K}:A_0\ot_k \ovl K\rightarrow A_{\ovl K},$ then
    \begin{flalign*}
    \xymatrix@C+4pc{ A_0\ot_k\ovl K  \ar[r]^{(\mr{tr}(\phi_{\mr{ab}}))|_{A_0})\ot_k\ovl K} \ar[d]^{h_{\mr{ab}}} &  A_0'\ot_k\ovl K\ar[d]^{h_{\mr{ab}}'} \\
    A_{\ovl K}\ar[r]^{\phi_{\mr{ab}}} &A'_{\ovl K}}
    \end{flalign*}

    Let $h_{\mr{tor}}=\mr{id}_{\mbb G_m^t}$ and $\psi_{\mr{tor}}=\phi_{\mr{tor}}.$ Since $k$ is algebraically closed, there exists $\underline \eta_0\in \mr{Ext}^1_{k}(A_0,\mathbb G_m^{t})$ such that $\psi_{\mr{tor},*}(\underline \eta_0)=\mr{tr}(\phi_{\mr{ab}})^*(\underline \eta_0')$. 
    By diagram chasing, we verify:
    \begin{align*}
        &\psi_{\mr{tor},*}(\underline\eta_0\ot_k\ovl K)=\mr{tr}(\phi_{\mr{ab}})^*(\underline \eta_0'\ot_k\ovl K)=\mr{tr}(\phi_{\mr{ab}})_{\ovl K}^*{h'_{\mr{ab}}}^*(\underline \eta')
        \\&\kern 1em=\mr{tr}|_{A_0\ot_k \ovl K}^*\phi_{\mr{ab}}^*(\underline \eta')
        =\mr{tr}|_{A_0\ot_k\ovl K}^*\phi_{\mr{tor},*}(\underline\eta)
        =\phi_{\mr{tor},*}\mr{tr}|_{A_0\ot_k \ovl K}^*(\underline\eta)
        =\phi_{\mr{tor},*}h_{\mr{ab}}^*(\underline \eta).
    \end{align*}
    Since $\psi_{\mr{tor}}$ is an isogeny, the kernel of $\psi_{\mr{tor},*}:\mr{Ext}^1_{\ovl K}(A_0\ot_k\ovl K,\mbb{G}_m^t)\rightarrow \mr{Ext}^1_{\ovl K}(A_0\ot_k\ovl K,\mbb{G}_m^t)$ is finite. Since the torsion points of $\mr{Ext}^1_{\ovl K}(A_0\ot_k\ovl K,\mbb{G}_m^t)\simeq (A_0^\vee)^t_{\ovl K}$ coincide with torsion points in $\mr{Ext}^1_{k}(A_0,\mbb{G}_m^t)\simeq (A_0^\vee)^t$, we can adjust $\underline \eta_0$ such that $\eta_0\ot\ovl K=h^*_{\mr{ab}}(\underline \eta).$
    Let $G_0$ be the semiabelian variety defined by $\underline \eta_0.$
    We thus obtain a homomorphism $h:G_0\ot_k \ovl K\rightarrow G_{\ovl K},$ induced by $h_{\mr{tor}}$ and $h_{\mr{ab}}.$
    Then $h$ has finite kernel, and the isogeny $$\phi:G_0\rightarrow G_0'$$ given by $\psi_{\mr{tor}}$ and $\psi_{\mr{ab}}:=\mr{tr}(\phi_{\mr{ab}})|_{A_0}$ satisfies $\phi\circ h=h'\circ (\psi\ot_k \ovl K)$.
    This concludes the proof.
\end{proof}
\subsection{Special subvarieties}\label{subsec_special}
Let $X\subset G_{\ovl K}$ be a closed subvariety. Recall that we say $X$ is special if
there exist
        \begin{itemize}
            \item a torsion point $x$ in $\widetilde G=G_{\ovl K}/\mr{Stab}_0(X)$,
            \item a semiabelian variety $G_0$ over $k$ with a homomorphism $h:G_0\ot_k \ovl K\rightarrow \widetilde G_{\ovl K}$ of finite kernel,
            \item a closed $k$-subvariety $X_0\subset G_0$,
        \end{itemize}  
        such that $X/\mr{Stab}_0(X)=h(X_0\ot_{k}\ovl K)+x.$ Here we recall that  $\mr{Stab}_0(X)$ is the irreducible component of $\mr{Stab}(X)$ which contains the identity, as explained in the introduction. 
\begin{prop}\label{prop_special_isogeny}
    Let $\phi:G_{\ovl K}\rightarrow G'_{\ovl K}$ be an isogeny of semiabelian varieties. Let $X\subset G_{\ovl K}$ be a closed subvariety, then $X$ is special if and only if so is $\phi(X)$.
\end{prop}
\begin{proof}
    Let $Y:=\phi(X).$ Define $\widetilde G:=G/\mr{Stab}_0(X)$ (resp. $\widetilde G':=G/\mr{Stab}_0(Y)$), and $\widetilde X:=X/\mr{Stab}_0(X)$ (resp. $\widetilde Y:=Y/\mr{Stab}_0(Y)$).
    Note that the homomorphism $\widetilde\phi:\widetilde G\rightarrow \widetilde G'$ is an isogeny. Thus, by definition, $Y$ is special if so is $X$.

    We may assume that there exists a semiabelian variety $H$ over $k$, and a morphism $f:H\ot_k\ovl K\rightarrow G'_{\ovl K}$ with finite kernel, such that $\widetilde Y=f(Y_0\ot_k\ovl K)$ for some closed subvariety $Y_0\subset H$. Applying Proposition \ref{prop_trace_isogeny} to $\widetilde\phi$, we obtain $h,h',G_0,G_0',\psi$.
    We may thus assume that $Y_0\subset G_0'$ by replacing $Y_0$ with its image in $G_0'.$

    Therefore $$h(\psi^{-1}(Y_0)\ot_k\ovl K)\subset \widetilde\phi^{-1}(h'(Y_0\ot_k \ovl K))=\widetilde\phi^{-1}(\widetilde Y),$$
    Hence there is an irreducible component $X_0$ of $h(\widetilde \phi)^{-1}(Y_0)$ such that $h(X_0\ot_k\ovl K)$ is an irreducible component of $\widetilde \phi^{-1}(\widetilde Y)$. On the other hand, since $\widetilde X$ is also an irreducible component of $\widetilde \phi^{-1}(\widetilde Y)$, there exists a torsion point $x\in \widetilde G$ such that $\widetilde X=h(X_0\otimes_k \ovl K)+x$.

\end{proof}

\subsection{Essential minimum}
Let $X\subset G_{\ovl K}$ be a closed subvariety.
The essential minimum of $L$ on $X$ is defined as 
$$\mr{ess}_{\ovl L}(X):=\mr{ess}_{\ovl L}(\ovl X)$$
where $\ovl X$ is the Zariski closure of $X$ in $\ovl G_{\ovl K}$.
Hence $\mr{ess}_{\ovl L}(X)\geq 0$, and equality holds if and only if $X$ contains Zariski dense sets of points of arbitrarily small canonical heights.
\begin{lemm}\label{lemm_ess_isogeny} Let $\phi:G\rightarrow G'$ be a homomorphism of semiabelian varieties. 
Let $N'$ be a symmetric ample line bundle on $A'$ and $M'$ be line bundle associated with the boundary divisor on $\ovl G'$. Let $L'=M'+N'$ be the ample line bundle on $\ovl G'$.
Let $Y=\phi(X)$. Then the following statements hold:
    \begin{enumerate}
        \item[\textnormal{(1)}] $\mr{ess}_{\ovl L'}(Y)=0$ if $\mr{ess}_{\ovl L}(X)=0$. 
        \item[\textnormal{(2)}] If $\phi$ is an isogeny, then $\mr{ess}_{\ovl L}(X)=0$ if $\mr{ess}_{\ovl L'}(Y)=0$. 
    \end{enumerate}
\end{lemm}
\begin{proof}
(1) 
We have a decomposition $\phi=\phi_1\circ\phi_2$ where $\phi_1:G\rightarrow G_1$ is given by the pair $(\phi_{\mr{tor}},\mr{id}_A)$, and $\phi_2:G_1\rightarrow G'$ is given by the pair $(\mr{id}_{\mbb G_m^{t'}},\phi_{\mr{ab}}).$

Since $N$ is ample, $\phi_{\mr{ab}}^*N'\leq c_2N$ for some $c_2>0$. Hence $c_2 h_{\ovl N}(\ovl \pi(x))\geq h_{\ovl N'}(\ovl \pi'(\phi_2(x)))$ for $x\in \ovl G_1(\ovl K).$ 

Note that the line bundle associated with the boundary divisor on $\ovl G_1$ is exactly $\phi_2^* M'.$
We then use the same argument as in \cite[Lemme 5.2]{Remond2003Approx}. Consider the standard coordinates $X_1,\cdots,X_t$ of $\mbb G_m^t$, and $Y_1,\cdots, Y_{t'}$ of $\mbb G_m^{t'}.$ Assume that $\phi_{\mr{tor}}^*Y_i=\prod_{j}X_j^{a_{i,j}},$ where $a_{i,j}\in\Z.$ Therefore $$h_{\ovl M'}(\phi_1(x))\leq c_1 h_{\ovl M}(x)$$
for $x\in G(\ovl K)$ where $c_1:=\displaystyle\max_{1\leq i\leq t'}\left\{\sum_{1\leq j\leq t}\lvert a_{i,j}\rvert\right\}.$ Hence $h_{\ovl L'}\leq \max\{c_1,c_2\}h_{\ovl L}$ on $G(\ovl K).$

(2) 
Since $\phi$ is an isogeny, after taking a finite extension of $K$, there exists an isogeny $\psi:G'\rightarrow G''$ of semiabelian varieties, such that $\psi_{\mr{ab}}=id_{A'}$ and $(\psi\circ \phi)_{\mr{tor}}=[n]:\mbb G_m^t\rightarrow \mbb G_m^t$ for some $n\in \N_+.$
Consider the isogeny $\psi\circ \phi:G\rightarrow G''$, which extends to $\ovl {\psi\circ \phi}:\ovl G\rightarrow \ovl G''.$

Set $N'':=N'$. Then $(\psi\circ \phi)_{\mr{ab}}^* N''$ is also ample, and $$\ovl{\psi\circ \phi}^*M''=nM$$ where $M''$ is the line bundle associated with the boundary divisor on $\ovl G''.$

We can take $c\gg 0$ such that $c\phi_{\mr{ab}}^* N''-N$ and $c\ovl \phi^* M''-M$ are effective. Let $\ovl L'':=\ovl M''+\ovl \pi''^*\ovl N''$ where $\ovl \pi'':\ovl G''\rightarrow A'$. Then $h_{\ovl L}(x)\leq ch_{\ovl L''}(\psi\circ\phi(x))$ for $x\in G(\ovl K).$

Applying the result in (1), we know that $h_{\ovl L''}$ can be bounded by a multiple of $h_{\ovl L'}$ on $G'(\ovl K).$ Eventually, $h_{\ovl L}$ can be bounded by a multiple of $h_{\ovl L'}$ on $ G(\ovl K)$.
\end{proof}

To sum up, Proposition \ref{prop_special_isogeny} and Lemma \ref{lemm_ess_isogeny} give the following.
\begin{prop}
    Let $\phi:G\rightarrow G'$ be an isogeny. The geometric Bogomolov conjecture holds for $G$ if and only if it holds for $G'$.
\end{prop}

We will postpone the proof that $X$ is special implies $\mr{ess}_{\ovl L}(X)=0$ to \S\ref{sec_REL_FZ}. Note that this is not as trivial as in the case of abelian varieties, since a special subvariety in a semiabelian variety may not have Zariski dense points of canonical height $0$. 

\section{Yamaki's relative height revisited}\label{sec_YMK_RVS}
We prove {\bf{(GBC)}} for the quasi-split case in this section, to which the general case will be reduced in \S\ref{sec_REL_FZ}.

Throughout this section, let $Y$ be a projective variety defined over $k$. Let $M_Y$ be a line bundle over $Y_K$. We say $M_Y$ is \emph{constant} if there exists a line bundle $M_0$ on $Y$ such that $M_Y=M_0\ot_k K$. Set that $\mc Y:=Y\times_k \mathfrak B$ and $\mc M_Y:=M_0\ot_k \mc O_{\mc Y}.$
We see that $(\mc Y,\mc M_Y)$ is a $\mathfrak B$-model of $(Y_K,M_Y)$. We denote by $\ovl {M_Y}$ the induced adelic line bundle. From now on, we assume that $M_0$ is ample on $Y$.
\begin{prop}\label{prop_trivial_gbc}
   Let $X\subset Y_{\ovl K}$ be a closed subvariety. If $h_{\ovl {M_Y}}(X)=0$ then $X=W_{\ovl K}$ for some closed $k$-subvariety $W\subset Y$. 
\end{prop}
\begin{proof}
    See \cite[Proposition 3.2]{Yamaki2018trace}. The proof actually works for general projective varieties.
\end{proof}
\begin{rema}
    If $\dim \mathfrak B=1$, the above proposition can be proved in a few words. Here we give the argument which shed a light to the original proof. Let $\mc X$ be the closure of $X$ in $\mc Y$, and $W=\mr{pr}_{Y}(X).$ It suffices to show that $\dim W= \dim X.$ Since $\ovl{M_Y}^{\dim X+1}\cdot [X]=M_Y^{\dim X+1}\cdot [W]$ by the projection formula, we obtain $\dim W<\dim X+1$ due to the ampleness of $M_Y.$
\end{rema}
\begin{exem}
    Let $Y=(\mbb P^1_k)^t,$ and $M_0=\underbrace{O(1)\boxtimes\cdots \boxtimes O(1)}_{t\text{ times}}.$ We call the induced adelic line bundle the \textit{canonical line bundle} on $(\mbb P^1_k)^t.$
\end{exem}
\subsection{Yamaki's relative height}
Let $A$ be an abelian variety over $K$.
Consider a closed subvariety $X\subset Y_{\ovl K}\times A_{\ovl K}=Y\times_k A_{\ovl K}$ of dimension $d$. We may further assume that $p_1(X)=Y_{\ovl K}$ where $p_1:Y_{\ovl K}\times A\rightarrow Y_{\ovl K}.$

For any $y\in Y$, we take an algebraic closure $\overline {\kappa(y)}$ of the residue field $\kappa(y)$ of $y$.
We denote by $\rho_y$ the generic point of $\ovl{\kappa(y)}\otimes_k\mathfrak B$, whose residue field is isomorphic to $\ovl{\kappa(y)}\otimes_k K$. To be more precise, $\rho_y$ gives a morphism $$\rho_y:\mathrm{Spec}(\ovl{\kappa(y)}\otimes_k K)\rightarrow Y$$
which factors through $Y\times_k \mathfrak B\rightarrow Y$.

Following \cite{Yamaki2018trace}, we introduce this ad hoc notation $$Y^{\mr{pd}}:=\{y\in Y\mid X_{\rho_y}\text{ is of pure dimension }\dim X-\dim Y\}.$$
We define a height function for all $y\in Y^{\mr{pd}}$.

Let $N$ be a symmetric ample line bundle on $A$, and let $\ovl N$ be the associated canonical line bundle.
We assume that $\ovl N$ is defined by a sequence $\{(\scrA_i\rightarrow \mathfrak B,\mc N_i,\ell_i)\}_{i\in \N}$.
Obviously, the sequence $$\{(\scrA_i\otimes_k\ovl{\kappa(y)}\rightarrow \mathfrak B\ot_k \ovl{\kappa(y)},\mc N_i\ot_k \ovl{\kappa(y)},\ell_i)\}_{i\in\N}$$ gives a defining sequence of the canonical height on $A_{\ovl{\kappa(y)}\ot_k K}$ over the model $(\mathfrak B\otimes_k \ovl{\kappa(y)},\mathcal H\ot_k \ovl{\kappa(y)}).$
The obtained adelic line bundle is denoted as $\ovl {N_y}$, where $N_y:=N\ot_K (\ovl {\kappa(y)}\ot_k K).$
\begin{defi}
    Let $y\in Y^{\mr{pd}}$. The \emph{relative height} $h_{X/Y}^{\ovl N}(y)$ is defined as $h_{\ovl{N_y}}(X_{\rho_y}).$
\end{defi}

The main result of this subsection is as follows:
\begin{prop}\label{prop_generic_height}
    Let $\eta$ be the generic point of $Y$. If $h_{\ovl {M_{Y}}\boxtimes\ovl N}(X)=0$, then $h_{X/Y}^{\ovl N}(\eta)=0.$
\end{prop}
    Before the proof, we provide the geometric interpretation of $h^{\ovl N}_{X/Y}(y)$. For each $y\in Y$, we denote the geometric point $\mr{Spec}(\ovl {\kappa(y)})\rightarrow Y$ by $\ovl y$.
    Note that $\mc A_i\ot_k\ovl{\kappa(y)}\simeq \{\ovl y\}\times \mc A_i$.
    Let $\mc X_i$ be the closure of $X$ in $Y\times_k\mc A_i$. Then $\mc X_{i}|_{\{\ovl y\}\times \mc A_i}$ gives the closure of $X_{\rho_y}$ in $\mc A_i\ot_k\ovl {\kappa(y)}$.
    For each $y\in Y^{\mr{pd}}$,
    we define $$h_i(y):=\frac{\deg_{\mc O_{Y}\boxtimes \mc N_i|_{\{\ovl y\}\times \mc A_i}}(\mc X_i|_{\{\ovl y\}\times \mc A_i})}{(d-\dim Y+1)\ell_i^{d-\dim Y+1}\deg_{N_y}(X_{\rho_y})}$$
    By definition, we have 
    $$h_{X/Y}^{N}(y)=\lim_{i\rightarrow+\infty}h_i(y)$$
    for every $y\in Y^{\mr{pd}}.$
    
    Using the generic flatness and reproducing the proof of \cite[Proposition 4.9]{Yamaki2018trace}, we will obtain the following.
    \begin{lemm}
    For each $i$, there exists a Zariski dense open subset $U_i\subset Y$ such that
    $h_i(y)$ is constant for $y\in U_i$.
    \end{lemm} 
\begin{proof}[Proof of Proposition \ref{prop_generic_height}]
    For each $i$, we denote by $\ovl N_i$ the adelic line bundle induced by $(\scrA_i\rightarrow \mathfrak B,\mc N_i,\ell_i)$. We follow the same argument as in \cite[Proposition 5.1]{Yamaki2018trace}.
        By definition,
        $${\frac{1}{\ell_i^{d+1-\dim Y}}}p_1^* M_0^{\dim Y}\cdot p_2^*\mc N_i^{d+1-\dim Y}\cdot[\mc X_i]=p_1^*\ovl{M_Y}^{\dim Y}\cdot p_2^*\ovl N_i^{d+1-\dim Y}\cdot[X].$$
        By Bertini's theorem, there exists an open subset $V\subset H^0(Y,M_0)^{\dim Y}$ such that for any $(s_1,\cdots, s_{\dim Y})\in V$, $\mathrm{div}(s_1)\cap\cdots\cap\mathrm{div}(s_{\dim Y})\subset U_i$, and is of pure dimension $0$. 
        Assume that $\mr{div}(s_1)\cap\cdots\cap \mr{div}(s_{\dim Y})=\sum \lambda_j [y_j]$
        where $\lambda_j\in\N$ and $y_j\in Y(k).$
        Then \begin{align*}
            &p_1^* M_0^{\dim Y}\cdot p_2^*\mc N_i^{d+1-\dim Y}\cdot[\mc X_i]=\sum \lambda_j p_2^*\mc N_i^{d+1-\dim Y}\cdot[\mc X_i|_{\{y_j\}\times \mc A_i}]
            \\&\kern 5em= \deg_{M_0}(Y)\deg_{N_\eta}(X_{\rho_\eta})\ell_i^{d+1-\dim Y}h_i(\eta).
        \end{align*}

    Therefore \begin{align*}
        &\kern 1.3em\deg_{M_0}(Y)\deg_{N_\eta}(X_{\rho_\eta})h_{X/Y}^{\ovl N}(\eta)
        =\deg_{M_0}(Y)\deg_{N_\eta}(X_{\rho_\eta})\lim_{i\rightarrow +\infty}h_i(\eta)
        \\&\kern 7em= \lim_{i\rightarrow +\infty}p_1^*\ovl {M_Y}^{\dim Y}p^*_2\ovl N_i^{d+1-\dim Y}\cdot[X]
        \\&\kern 7em\leq \lim_{i\rightarrow +\infty}\frac{1}{\binom{d+1}{\dim Y}}(\ovl {M_Y}\boxtimes \ovl N_i)^{d+1}\cdot[X]=0,
    \end{align*}
    where the last inequality is due to the nefness of $\ovl{M_Y}$ and $\ovl N_i$.
    \end{proof}
\subsection{Quasi-split case}
The following is a generalization of \cite[Theorem 5.3]{Yamaki2018trace}.
\begin{prop}\label{prop_trace_gbc}
    Assume that the $\ovl K/k$-trace of $A$ is trivial. 
    Let $X\subset Y\times_k A_{\ovl K}$ be a closed subvariety. If $h_{\ovl {M_Y}\boxtimes \ovl N}(X)=0$, then there exists \begin{enumerate}
        \item[\textnormal{(i)}] a closed $k$-subvariety $W\subset Y$,
        \item[\textnormal{(ii)}] an abelian subvariety $A'\subset A_{\ovl K}$,
        \item[\textnormal{(iii)}] torsion point $x\in A_{\ovl K}$,
    \end{enumerate} such that $$X=W\times_k(A'+x).$$
\end{prop}
\begin{proof}
    Let $e=\dim p_1(X)$, where $p_1:Y\times_k A_{\ovl K}\rightarrow Y_{\ovl K}.$
    By the projection formula, $$p_1^*\ovl {M_Y}^{e+1}\cdot p_2^*\ovl N^{\dim X-e}\cdot [X]=\deg_{N_{\eta}}(X_{\eta})\ovl {M_Y}^{e+1}\cdot [p_1(X)],$$
    where $\eta$ is the generic fiber of $p_1(X),$ and $N_\eta:=N\ot_K\kappa(\eta)$. By the nefness of $\ovl {M_Y}$ and $\ovl N$, the left-hand side is bounded from above by $(\ovl {M_Y}\boxtimes \ovl N)^{\dim X+1}\cdot[X],$ hence is $0.$ Therefore $p_1(X)=W_{\ovl K}$ for some closed $k$-subvariety $W\subset Y$ by Proposition \ref{prop_trivial_gbc}.

    By abuse of notation, we still denote by $\eta$ the generic point of $W$. Let $K':=\ovl{\kappa(y)}\ot_k K$. Since $A_{K'}$ is of trivial $\ovl K'/\ovl{\kappa(y)}$-trace by \cite[Lemma A.1]{Yamaki2018trace}, applying {\bf(GBC)} for abelian varieties \cite[Theorem 1.3]{YX2022GBC}, $X_{\rho_\eta}$ is a torsion subvariety in $A_{\ovl K'}$ of dimension $\dim X-e$. By \cite[Proposition 3.7]{Yamaki2018trace}, there exist an abelian subvariety $A'\subset A_{\ovl K}$ and a torsion point $x\in A_{\ovl K}$, such that $$X_{\rho_\eta}=(A'+x)\ot_{\ovl K}{\ovl K'}.$$
    Therefore $W\times_k(A'+x)\subset X$. By dimension reason, we obtain the equality.
\end{proof}

We say $G$ is \textit{quasi-split} if $G$ is isogenous to $G_0\times_k A$ where $A$ is of trivial $\ovl K/k$-trace.
Applying Proposition \ref{prop_trace_gbc}, we obtain the following.
\begin{coro}\label{coro_quasi-split}
    Assume that $G$ is quasi-split. Let $X$ be a closed subvariety of $G_{\ovl K}$ containing Zariski dense sets of small points. Then $X$ is special.
\end{coro}
\section{Proof of the geometric Bogomolov conjecture}\label{sec_REL_FZ}
\subsection{An auxiliary lemma}

The following Lemma is the key ingredient of this section, connecting the vanishing of the essential minimum with certain intersection numbers.
\begin{lemm}\label{lemm_aux}
Let $\mbb{G}_m^{t'}$ be a torus, and $X\subset \mbb{G}_{m}^{t'}\times G_{\ovl K}$ be a closed subvariety. Let $\ovl{M}'$ be the canonical line bundle on $(\mbb{P}^1_K)^{t'}$, and $\ovl L':=\ovl M' \boxtimes \ovl L$. Then $\mr{ess}_{\ovl L'}(X)=0$ if and only if
$$p_2^*\ovl{M}^{\dim X+1-\dim p(X)}\cdot p^*(\ovl M'\boxtimes \ovl{N})^{\dim p(X)}\cdot [\ovl X]=0,$$
and $$(\ovl M'\boxtimes \ovl N)^{e+1}\cdot [\ovl{p(X)}]=0$$
where $p_2:(\mbb P^1)^{t'}\times \ovl G\rightarrow \ovl G$ and $p:(\mbb P^1)^{t'}\times \ovl G\rightarrow (\mbb P^1)^{t'}\times A.$
\end{lemm}

\begin{proof}
Consider the morphism $f_n:(\mbb P^1)^{t'}\times \ovl G\rightarrow (\mbb P^1)^{t'}\times\ovl G$ extending $[n^2]\times[n]:\mbb{G}_{m}^{t'}\times G\rightarrow \mbb{G}_{m}^{t'}\times G.$
Then
$$f_n^*\ovl{L}'=n^2p^*(\ovl{M}'\boxtimes\ovl{N})+np_2^*\ovl{M}.$$
The fundamental inequality gives that
\begin{equation}\label{ineq_fund_[n2,n]_dyn}
    \mr{ess}_{f^*_n\ovl L'}(X)\ge h_{f_n^*\ovl L'}(X)\ge\frac{1}{d+1}(\mr{ess}_{f^*_n\ovl L'}(X)+d\cdot\lambda_{\min}(f_n^*\ovl L'|_X))
\end{equation}
where $d:=\dim X$.
\begin{align*}
    &(f_n^*\ovl{L}')^{d+1}\cdot [\ovl X]\\&\kern 3em=\sum_{0\leq j\leq e+1} \binom{d+1}{j}n^{d+1+j} p^*(\ovl M'\boxtimes \ovl N)^j\cdot p_2^*\ovl M^{d+1-j}\cdot [\ovl X]
\end{align*}
where $e:=\dim p(X).$
Assume that $\mr{ess}_{\ovl L'}(X)=0$. Then $\mr{ess}_{\ovl M'\boxtimes \ovl N}(p (X))=0$ in $\mbb G_m^{t'}\times A,$ which implies that $$(\ovl M'\boxtimes \ovl N)^{e+1}\cdot [\ovl{p(X)}]=0$$ by the nefness of $\ovl M'$ and $\ovl N$.
By the projection formula, we have $$(f_n^*\ovl L')^{d+1}\cdot[\ovl X]=n^{d+1+e}\binom{d+1}{e}p^*(\ovl M'\boxtimes \ovl N)^e\cdot p_2^*\ovl M^{d+1-e}\cdot[\ovl X]+O(n^{d+e}).$$
On the other hand, we can expand to get
$$\deg_{f^*_n {L'}}(\ovl X)=n^{d+e}\binom{d}{e} p^*(M'\boxtimes N)^e\cdot p_2^* M^{d-e}\cdot [\ovl X]+O(n^{d+e-1}).$$
Hence \begin{equation}\label{eq_limit_[n2,n]_dyn}
    \lim_{n\rightarrow \infty}\frac{1}{n} h_{f_n^*\ovl L'}(X)=\frac{p^*(\ovl M'\boxtimes \ovl N)^e\cdot p_2^*\ovl M^{d+1-e}\cdot[\ovl X]}{(d+1-e) p^*(M'\boxtimes N)^e\cdot p_2^* M\cdot[\ovl X]}.
\end{equation}
Note that $\mr{ess}_{f^*_n\ovl L'}(X)\leq n^2\mr{ess}_{\ovl L'}(X)=0$ and \begin{equation}\label{ineq_abs_[n2,n]_dyn}
    \lambda_{\min}(f^*_n\ovl L'|_X)\geq\lambda_{\min}(f^*_n\ovl L')\geq \lambda_{\min}(\ovl L')>-\infty
\end{equation}
by Proposition \ref{prop_ample_fin_pullback}.
Dividing \eqref{ineq_fund_[n2,n]_dyn} by $n$, and letting $n\rightarrow+\infty$, we obtain that $$p^*(\ovl M'\boxtimes \ovl N)^e\cdot p_2^*\ovl M^{d+1-e}\cdot[\ovl X]=0.$$

Conversely, assume that $p^*(\ovl M'\boxtimes \ovl N)^e\cdot p_2^*\ovl M^{d+1-e}\cdot[\ovl X]=0$ and $(\ovl M'\boxtimes \ovl N)^{e+1}\cdot [\ovl{p(X)}]=0.$ By \eqref{eq_limit_[n2,n]_dyn}, $\displaystyle\lim_{n\rightarrow \infty}\frac{1}{n} h_{f_n^*\ovl L'}(X)=0$. Therefore 
$$0=\lim_{n\rightarrow +\infty}-\frac{d\cdot \lambda_{\min}(\ovl L')}{n(d+1)}\geq \lim_{n\rightarrow+\infty}\frac{\mr{ess}_{f^*_n\ovl L'}(X)}{n(d+1)}\geq \frac{\mr{ess}_{\ovl L'}(X)}{d+1}$$
where the first inequality is due to \eqref{ineq_fund_[n2,n]_dyn} and \eqref{ineq_abs_[n2,n]_dyn}, and the second inequality is due to the fact that $f_n^*\ovl L'-n\ovl L'$ is effective. This concludes the proof.
\end{proof}
\begin{rema}
    The number field analogue of Lemma \ref{lemm_aux} generalizes \cite[Lemme 3, Lemme 4]{DP2000Semiabel}. Moreover, our proof only uses the fundamental inequality, avoids any reliance on the Manin–Mumford conjecture, which is inapplicable in this context due to the potential non-density of height-zero points, as illustrated in Example \ref{exam_nondense}.
\end{rema}
\subsection{$X$ is special $\Rightarrow \mr{ess}_{\ovl L}(X)=0$}
In this subsection, we will show that a special subvariety contains Zariski dense sets of small points.
\begin{lemm}
    Let $Q=T_y^*N-N\in \mr{Pic}^0(A)$ for $y\in A(K)$. This gives a semiabelian variety $$0\rightarrow \mathbb G_m\rightarrow G\xrightarrow{\pi} A\rightarrow 0.$$ 
    For $x\in A(\ovl K)$, if $h_{\ovl N}(x)=0$, then we have $$\mr{ess}_{\ovl M}(\pi^{-1}(x))=0.$$ 
\end{lemm}

\begin{proof}
By Lemma \ref{lemm_aux}, it suffices to prove
$$\ovl M^2\cdot[\ovl{\pi}^{-1}(x)]=0.$$
Recall the sequence $(\ovl{\mc G}_m,[m]^*\mc D_{[\infty]},m)$ (resp. $(\ovl{\mc G}_m,[m]^*\mc D_{[\infty]},m)$) from Theorem \ref{theo_eff_boud_div}, defining the adelic divisor $\ovl D_{[\infty]}$ (resp. $\ovl D_{[0]}$). Since $[m]^*\mc D_{[\infty]}$ and $[m]^*\mc D_{[0]}$ have disjoint support, we have
$$\ovl D_{[\infty]}\cdot\ovl D_{[0]}\cdot[\ovl{\pi}^{-1}(x)]=\lim_{m\to\infty}\frac{1}{m^2}([m]^*\mc D_{[\infty]})\cdot([m]^*\mc D_{[0]})\cdot\mc Y_m=0,$$
where $\mc Y_m$ is the Zariski closure of $\ovl{\pi}^{-1}(x)$ in $\ovl{\mc G}_m$. So
$$\ovl M^2\cdot[\ovl{\pi}^{-1}(x)]=(\ovl D_{[\infty]}+\ovl D_{[0]})^2\cdot[\ovl{\pi}^{-1}(x)]=(\ovl D_{[\infty]}-\ovl D_{[0]})^2\cdot[\ovl{\pi}^{-1}(x)].$$
By the projection formula,
$$(\ovl D_{[\infty]}-\ovl D_{[0]})^2\cdot[\ovl{\pi}^{-1}(x)]=\pi^*\ovl Q^2\cdot[\ovl{\pi}^{-1}(x)]=0.$$
\end{proof}

\begin{prop}
Let $G$ be a semiabelian variety, $H\subset G$ a semiabelian subvariety and $\phi:G\to G/H$ the quotient homomorphism. Let $X\subset G/H$ be a closed subvariety. If $\mr{ess}_{\ovl{L}'}(X)=0$, then $\mr{ess}_{\ovl{L}}(\phi^{-1}(X))=0.$
\end{prop}

\begin{proof}
If $H$ is an abelian variety, i.e., $\pi|_H:H\rightarrow A$ is an injective homorphism of abelian varieties. By the Poincaré reducibility theorem, there is a quotient homomorphism $\psi:A\to H$ such that the composition $H\to A\to H$ is an isogeny. So $G$ is isogenous to $H\times\ker(\psi\circ\pi)$ and $\ker(\psi\circ \pi)\to G\to G/H$ is an isogeny. By Lemma \ref{lemm_ess_isogeny} (2), it suffices to consider the case $G=G/H\times H$, which is trivial.  

If $H=\mathbb{G}_m$, then $G/H$ corresponds to $\underline\eta'\in \mr{Ext}_K^{1}(A,\mathbb{G}_m^{t-1})$. We may take $\eta_0\in\mr{Ext}_K(A,\mbb{G}_m)$ such that, after taking finite extension of $K$, $(\eta_0,\underline \eta')\in \mr{Ext}^{1}_K(A,\mathbb{G}_m^t)$ gives a semiabelian variety isogenous to $G$. Again by Lemma \ref{lemm_ess_isogeny} (2), we may assume that $G$ is given by $(\eta_0,\underline \eta').$ Let $Q_0\in \mr{Pic}^0(A)$ correspond to $\eta_0.$ Then $\ovl G=\mbb P(\mc O_A\oplus Q_0^\vee)\times_A\ovl{G/H},$ and $\phi$ extends to $\ovl \phi:\ovl G\rightarrow \ovl{G/H}.$

Let $\ovl M_0$ (resp. $\ovl M'$) be the canonical line bundle associated with the boundary divisor on $\mbb P(\mc O_A\oplus Q_0^\vee)$ (resp. $\ovl{G/H}$). By Lemma \ref{lemm_aux}, it suffices to prove that 
$$(p^*\ovl M_0+\ovl\phi^*\ovl M')^{e+2}\cdot \ovl \pi^*\ovl N^{d-e}\cdot[\ovl{\phi^{-1}(X)}]=0,$$
where $p:\ovl G\rightarrow \mbb P(\mc O_A\oplus Q_0^\vee)$, $d=\dim X$ and $e=\dim X-\dim(\pi(X)).$

Considering the projections $\ovl {\phi^{-1}(X)}\xrightarrow{\ovl\phi} \ovl X$ and $\ovl {\phi^{-1}(X)}\xrightarrow{p}\ovl \pi_0^{-1}(\ovl\pi(\ovl {\phi^{-1}(X)}))$
where $\ovl \pi_0:\mbb P(\mc O_A\oplus Q_0^\vee)\rightarrow A,$
we see that 
$$F(i):=p^*\ovl M_0^{e+2-i}\cdot \ovl \phi^*\ovl M'^{i}\cdot\ovl \pi^*\ovl N^{d-e}\cdot[\ovl{\phi^{-1}(X)}]$$
vanishes for $i\not=e,e+1$ due to dimension reasons. Since $\mr{ess}_{\ovl L'}(X)=0$, $\ovl M'^{e+1}\cdot \ovl \pi'^*\ovl N^{d-e}\cdot[\ovl X]=0$ where $\ovl \pi':\ovl{G/H}\rightarrow A.$ Hence $F(e+1)=0$.

Notice that $\pi'(X)$ contains a Zariski dense set of points of height $0$. We see $\mr{ess}_{\ovl M_0+\ovl \pi_0^*\ovl N}(\ovl \pi_0^{-1}(\ovl {\pi'(X)}))=0,$ whence $$\ovl M_0^2\cdot\ovl \pi_0^*\ovl N^{d-e}\cdot[\ovl \pi_0^{-1}(\ovl {\pi'(X)})]=0.$$ Again by the projection formula, $F(e)=0.$

For a general $H$, given by the exact sequence:
$$0\to\mathbb{G}_m^{t'}\to H\to B\to 0,$$
we reduce step by step via intermediate quotients $G_i=G/\mathbb{G}_m^i$, and apply the above cases inductively. 

\end{proof}

\begin{coro}
    If $X\subset G_{\ovl K}$ is special, then $\mr{ess}_{\ovl L}(X)=0.$
\end{coro}

\subsection{Criterion for the quasi-splitness}\label{subsec_reduction}
From now on, we assume that $A=A_0\times_k A_1$ where $A_0$ is an abelian variety over $k$, and $A_1$ is an abelian variety over $K$ with trivial $\ovl K/k$-trace. We further assume that $N=(N_0\ot_k K)\boxtimes N_1$ where $N_0$ and $N_1$ are symmetric ample line bundles on $A_0$ and $A_1$ respectively. 
By abuse of notation, we denote by $\ovl N_0$ the canonical bundle of adelic lines associated with $N_0\ot_k K$ on $A_0\ot_k K$. Then $\ovl N=\ovl N_0\boxtimes \ovl N_1.$

Recall that after taking a finite field extension, our semiabelian variety is associated with $Q^{(1)},\dots,Q^{(t)}\in\mr{Pic}^0(A)$ by the Weil-Barsotti formula as in \S\ref{subsec_height_semi}. For each $i=1,\dots, t$, $Q^{(i)}= Q_0^{(i)}\boxtimes Q_1^{(i)}$, where $Q_0^{(i)}\in \mr{Pic}^0(A_0\ot_k K)$ and $Q_1^{(i)}\in \mr{Pic}^0(A_1)$.

Note that by definition, the canonical line bundle associated to $Q_0^{(i)}$ is induced by a model $(A_0\times \mathfrak B, \mc Q_0^{(i)})$. $Q_0^{(i)}$ is constant if and only if $\mc Q_0^{(i)}=\mr{pr}_{A_0}^*{Q'}$ for some $Q'\in \mr{Pic}^0(A_0).$

We see that $G$ is quasi-split if and only if $Q_0^{(i)}$ are constant and $Q_1^{(i)}$ are torsion. In this subsection, we will give a criterion, eventually showing that $G$ is quasi-split if it contains a "large" closed subvariety which is of essential minimum $0.$

\begin{theo}\label{theo_constant_torsion_test}
    Let $t'\geq 0$. We denote by $\ovl M'$ the canonical line bundle on $(\mbb P^1)^{t'}$.
    Let \begin{align*}
        &p_2:(\mbb P^1)^{t'}\times \ovl G\rightarrow \ovl G\text{ and }
        \\&p:(\mbb P^1)^{t'}\times \ovl G\rightarrow (\mbb P^1)^{t'}\times A.
    \end{align*}
    Let $X\subset \mbb G_m^{t'}\times_k \ovl G$ be a closed variety such that
    $p|_{\ovl X}:\ovl X\rightarrow Y\times_k A_1$ is generically finite of degree $r>0$, for some closed $k$-subvariety $Y\subset (\mbb P^1)^{t'}\times_k A_0$. We assume that \begin{itemize}
        \item the projection of $Y$ to $A_0$ generates $A_0,$ and 
        \item $\displaystyle p_2^*\ovl M\cdot p^*(\ovl M'\boxtimes \ovl N)^{\dim X}\cdot [\ovl X]=0.$
    \end{itemize} 
    Then $G$ is quasi-split.
\end{theo}
\begin{proof}
     We may assume that $Y$ is normal. Indeed, one may replace $Y$ by its normalization and take the base change of $X$.
     We incorporate the notations in Theorem \ref{theo_eff_boud_div}, and define $$\ovl E_{[0]}^{(i)}:=p_2^* \ovl D_{[0]}^{(i)}\cdot [X]\text{ and }\ovl E_{[\infty]}^{(i)}:=p_2^*\ovl D_{[\infty]}^{(i)}\cdot [X].$$ Here we give a brief explanation, since $X\not\subset p_2^{-1}([0]\cup[\infty])$, $p_2^*D_{[0]}^{(i)}\cdot[X]$ (resp. $p_2^*D_{[\infty]}^{(i)}\cdot[X]$) is a Cartier divisor on $X$, and the restriction of Green functions on $X$ are Green functions with respect to $p_2^*D_{[0]}^{(i)}\cdot[X]$ (resp. $p_2^*D_{[\infty]}^{(i)}\cdot[X]$), which gives to the adelic divisor $\ovl E_{[0]}^{(i)}$ (resp. $\ovl E_{[\infty]}^{(i)}$). As $\ovl D_{[0]}^{(i)}$ and $\ovl D_{[\infty]}^{(i)}$ are effective, $\ovl E_{[0]}^{(i)}$ and $\ovl E_{[\infty]}^{(i)}$ are also effective.

    Fix $i=1,\dots,t$.
    Notice that \begin{align*}
        &0=p_2^*\ovl M\cdot p^*(\ovl M'\boxtimes \ovl N)^{\dim X}\cdot [X]=p^*(\ovl M'\boxtimes \ovl N)^d\cdot\sum_{1\leq j\leq t} \left(\ovl E_{[0]}^{(j)}+\ovl E_{[\infty]}^{(j)}\right)\\
        &\kern 7em\geq (\ovl M'\boxtimes \ovl N)^d\cdot p_*\left(E_{[0]}^{(i)}+E_{[\infty]}^{(i)}\right)\geq 0.
    \end{align*}
    Applying Proposition \ref{prop_trivial_gbc} to each irreducible component of $p_*\left(E_{[0]}^{(i)}+E_{[\infty]}^{(i)}\right)$, we obtain that
    \begin{enumerate}
        \item a divisor $D_0$ (defined over $k$) on $Y,$ and 
        \item a divisor $D_1$ on $A_1$ whose irreducible components are torsion translates of abelian subvarieties in $A_1,$ 
    \end{enumerate}
    such that $$p_*(E_{[\infty]}^{(i)}- E_{[0]}^{(i)})=\mr{pr}_Y^*D_0+\mr{pr}_{A_1}^*D_1.$$
    
    Since $Y$ is normal, by Theorem \ref{theo_eff_boud_div}, $$p_*\left(E_{[\infty]}^{(i)}-E_{[0]}^{(i)}\right)=rc_1(\mr{pr}_A^*Q^{(i)})$$
    where $\mr{pr}_A:Y\times_k A_1\rightarrow A.$ 
    Hence there exists a rational section $s_0$ of $rQ_0^{(i)}\ot_{\mc O_{A_0}} \mc O_{Y_K}$, and 
    a rational section $s_1$ of $rQ_1^{(i)}$ such that $\mr{div}(s_0)=D_0\ot_k K$ and $\mr{div}(s_1)=D_1.$ 
    
    After taking pull-back by an isogeny, one may assume that $D_1=\sum \lambda_i [B_i]$, where $\lambda_i\in\Z$ and $B_i$ are abelian subvarieties of codimension $1$ in $A_1$. Hence $rQ_1^{(i)}$ is a $2$-torsion since it is both anti-symmetric and symmetric. From now on, we may assume that $Q_1^{(i)}$ are trivial for all $i=1,\dots, t$ after taking an isogeny of $G$.
    
    On the other hand, $rQ_0^{(i)}|_{Y_K}=\mc O_{Y}(D_0)\ot_k K$, which shows that $r\mc Q_0^{(i)}|_{Y\times \mathfrak B}=\mc O_{Y}(D_0)\boxtimes \mc F$ for some line bundle $\mc F$ on $\mathfrak B$. In fact, $\mathcal F$ has to be trivial. This can be done by a similar argument as above on the model $$\mbb P(\mc O_{A_0\times \mathfrak B}\oplus (\mc Q_0^{(1)})^\vee)\times_{A_0\times \mathfrak B}\cdots\times_{A_0\times \mathfrak B}\mbb P(\mc O_{A_0\times \mathfrak B}\oplus (\mc Q_0^{(t)})^\vee).$$

    As the projection of $Y$ to $A_0$ generates $A_0$, replacing $X$ by $[n]X$ for every $n\in\Z$, we obtain that for Zariski dense $x\in A_0(k)$, $r' r\mc Q_0^{(i)}|_{\{x\}\times \mathfrak B}$ is trivial, where $r'$ is the degree of the normalization of $Y$. By the seesaw theorem, $r'rQ_0^{(i)}$ is constant, which concludes the proof.

\end{proof}

\subsection{Applying the relative Faltings--Zhang maps}
Since $G$ is a $\mbb{G}_m^t$-torsor over $A$, we have an isomorphism
$$\beta_n:G_{/A}^n\to(\mbb{G}_m^t)^{n-1}\times G,$$
where $G_{/A}^n$ is $n$-fold product $G\times_A \cdots\times_A G.$
We define the relative Faltings--Zhang map as
$$\alpha_n=(\mr{id}_{(\mbb{G}_m^t)^{n-1}}\times\pi)\circ\beta_n:G\times_A  \cdots\times_A G\to(\mbb{G}_m^t)^{n-1}\times A.$$
Denote the $i$-th projection $$p_i:\ovl {G^n_{/A}}=\ovl G\times _A\cdots \times_A \ovl G\rightarrow \ovl G$$
and, by abuse of notation, we write $\pi:G^n_{/A}\rightarrow A$. Define $\ovl L_n:=\sum_{1\leq i\leq n} p_i^*\ovl M +\ovl \pi^*\ovl N,$ which is the canonical line bundle on $G^n_{/A}.$

Let $S=\pi(X),$ and $X^n_{/S}=\underbrace{X\times_S\cdots\times_S X}_{n\text{ times}}\subset G^n_{/A}.$
\begin{prop}\label{prop_ess_min_relFZ} Let $X'$ be an irreducible component of $X^n_{/S}.$
If $\mr{ess}_{\ovl{L}}(X)=0$, then $\mr{ess}_{\ovl{L}_n}(X')=0$.
\end{prop}

\begin{proof}
It suffices to show that
$$(\sum_{1\leq i\leq n} p_i^*\ovl M)^{ne+1}\cdot\ovl \pi^*\ovl N^{\dim \pi(X)}\cdot [\ovl X']=0$$
where $e=\dim X-\dim \pi(X).$
The left hand side is a linear combination of
$$p_1^*\ovl M^{j_1}\cdot p_2^*\ovl M^{j_2}\cdots p_n^*\ovl M^{j_n}\cdot\pi^*\ovl N^{\dim \pi(X)}\cdot [\ovl X']$$
for $j_1,\dots,j_n\geq 0$ and $j_1+j_2+\cdots+j_n=ne+1$. We will show that every term is $0$. Without loss of generality, we assume $j_n=\min j_i$. By the pigeon-hole principle, $j_n\leq e$

If $j_n\leq e-1$, consider the composition $X'\hookrightarrow X^n_{/S}\to X^{n-1}_{/S}$, where the second morphism omits the $n$-th coordinate. Since 
$$(ne+1-j_n)+\dim \pi(X)\geq (n-1)e+\dim \pi(X)+2>\dim X^{n-1}_{/S}+1,$$ the intersection number is $0$ by the projection formula.

If $j_n=e$, consider the morphism $X'\rightarrow X$. Again, by the projection formula, the intersection number is proportional to
$\ovl M^{e+1}\cdot\ovl \pi^*\ovl{N}^{\dim \pi(X)}\cdot [\ovl X],$
which is $0$ by Lemma \ref{lemm_aux}.
\end{proof}

\begin{prop}\label{prop_rel_FZ_gen_fin}
Assume that $X$ has trivial stabilizer. For $n$ large enough, there exists an irreducible component $X'$ of $X^n_{/S}$ such that the restriction of $\alpha_n$ is generically finite from $X'$ to its image.
\end{prop}

\begin{proof}
    Let $\eta$ be the generic point of $S$.
    Let $\ovl \eta$ be a geometric generic point of $S$. Then $G_{\ovl\eta}$ is a $\mbb{G}_{m}^t$-torsor. Let $X_0$ be an irreducible component of $X_{\ovl\eta}$. The stabilizer of $X_1\subset G_{\ovl\eta}$ is trivial since $\mr{Stab}(X)=0$. 
    Following the same argument as in \cite[Lemma 3.1]{Zhang1998Bogomolov}, 
    we obtain that the restriction of $\alpha_n$ to $X_0^n$ is birational to its image. Let $X'$ be the closure of $X_0^n$ in $G$, which is as desired.
\end{proof}

\begin{proof}[Proof of Theorem \ref{theo_GBC}]
Since $G_{\ovl K}/\mr{Stab}_0(X)\rightarrow G_{\ovl K}/\mr{Stab}(X)$ is an isogeny, after taking the quotient by $\mr{Stab}(X)$, 
we may assume that $\mr{Stab}(X)=0$ and $\mr{ess}_{\ovl L}(X)=0$.

We further assume that $A=A_0\times_k A_1$ as in Subsection \ref{subsec_reduction}.

We may take sufficiently large $n$ such that there exists an irreducible component $X'$ of $X^n_{/S}$, which is generically finite to $\alpha_n(X')$ as in Proposition \ref{prop_rel_FZ_gen_fin}. Notice that $\alpha_n(X')$ and its image $\beta_n(X')$ are of essential minimum $0$ due to Proposition \ref{prop_ess_min_relFZ} and Lemma \ref{lemm_ess_isogeny}. By Proposition \ref{prop_trace_gbc}, $\ovl {\alpha_n(X')}\subset (\mathbb P^1_k)^{t(n-1)}\times_k A_{\ovl K}$ is of form $Y\times_k T$ where $Y$ is a closed $k$-subvariety of $(\mathbb P^1_k)^{t(n-1)}\times_k A_0$, and $T$ is a torsion translate of an abelian subvariety in $A_1\ot_K\ovl K$. After taking an isogeny and shrinking $A$, we may assume that $T=A_1$ and the projection of $Y$ to $A_0$ generates $A_0$. By Lemma \ref{lemm_aux}, we are in the situation of Theorem \ref{theo_constant_torsion_test}. Therefore, $G$ is quasi-split. Finally, since $\mr{Stab}(X)=0$, $X$ is a torsion translate of a constant variety due to Corollary \ref{coro_quasi-split}.
\end{proof}

\section{Appendix A}\label{sec_Appendix A}
In this appendix, we will prove the modified fundamental inequality in Theorem \ref{theo_fund_ineq}. This addresses a technical subtlety: The canonical line bundle on the semiabelian variety is not a priori a limit of nef line bundles on models.
Our argument follows from \cite[\S 5]{zhangposvar}.

\begin{prop}\label{prop_ample_fin_pullback} Let $X$ be a projective variety over $K$, and $\ovl L$ be an adelic line bundle on $X$.
    \begin{enumerate}
        \item[\textnormal{(1)}] If $L$ is ample, then $\lambda_{\min}(\ovl L)>-\infty.$
        \item[\textnormal{(2)}] Let $f:Y\rightarrow X$ be a proper morphism of projective $K$-varieties. Then $\lambda_{\min}(f^*\ovl L)\geq \lambda_{\min}(\ovl L).$
    \end{enumerate}
\end{prop}
\begin{proof}
    (1) After taking a sufficiently large multiple, we may assume that $L$ is very ample. Then there exists $s_1,\dots,s_{\dim X}\in H^0(X,L)$ such that $\cap_i\lvert\mr{div}(s_i)\rvert=\emptyset.$ 
    Let $V_1,\dots,V_m$ be the irreducible components of $Z$. For each $i=1,\dots,m$, chose $s'_i\in \{s_1,\dots, s_{\dim X}\}$ such that $s_i'$ does not vanish on $V_i.$ 
    Since $L$ is ample, there exists $t_1,\dots,t_m\in H^0(X,nL)$ for $n\gg0$ such that $t_i|_{V_j}\begin{cases}
        =0 \text{ if }i\not=j\\
        \not=0\text{ if }i=j
    \end{cases}$. 
    Let $$\displaystyle C_\nu=\max_{1\leq i\leq \dim X}\{\lVert s_i\rVert^{\sup}_{\phi_\nu}\}\text{ and }\displaystyle D_\nu=\max_{1\leq i\leq m}\{\lVert t_i\rVert^{\sup}_{n\phi_\nu}\}.$$
    Then for any $\ell>0$ and $\nu\in \mathfrak B^{(1)}$, $$\lVert \sum (s'_i)^\ell \cdot t_i\rVert_{(\ell+n)\phi_\nu}^{\sup}\leq C_\nu^\ell\cdot D_\nu.$$
    Therefore $\lambda_{\min}(\ovl L)\geq-\sum_{\nu\in \mathfrak B^{(1)}}\deg_{\mc H}(\nu)\cdot\log C_\nu.$

    (2) For any $s\in H^0(X,L)$, by definition, $\lVert f^*s\rVert_{f^*\phi_\nu}(x)=\lVert s\rVert_{\phi_\nu}(f(x))$ for any $x\in X^\an_\nu,$ which implies that $\lVert f^*s\rVert_{f^*\phi_\nu}^{\sup}\leq\lVert s\rVert_{\phi_\nu}^{\sup}$. Hence $\widehat\deg_{f^*\ovl L}(f^*s)\geq\widehat\deg_{\ovl L}(s).$
    
    Let $Z\subsetneq Y$ be a closed subset with irreducible components $V_1,\dots,V_m$. If $f(V_i)=X$, then $f^*s$ does not vanish on $V_i$ for any non-zero $s\in H^0(X,L).$ Therefore we may assume that $f(V_i)\subsetneq X$ for $i=1,\dots, m$. For any $t< \lambda_{\min}(\ovl L)$, there exists $s\in H^0(X,nL)$ such that $\widehat\deg_{n\ovl L}(s)\geq nt$, and $s$ does not vanish on every $f(V_i).$ Since $f^*s$ does not vanish on $V_i$ for $i=1,\dots, m$, we conclude the proof.
\end{proof}

\begin{proof}[Proof of the fundamental inequality]
    The first inequality is due to \cite[Proposition 6.4]{CMArakelovAdelic} and \cite[Theorem B]{CM2024Positivity}.
    For the second inequality, we follow the argument of \cite[Theorem 5.2]{zhangposvar} proceeding by induction on $\dim X$. If $\dim X=0$, there is nothing to prove. Now assume that for any projective variety of dimension smaller than $\dim X$, the fundamental inequality holds.
    
    For any $\epsilon>0$,
    by definition, there exists a proper closed subset $Z\subset X$ such that $$\mr{ess}_{\ovl L}(X)\leq \inf_{x\not\in Z(\ovl K)}h_{\ovl L}(x)+\epsilon.$$
    After taking sufficiently large multiple of $L$, we may assume that there exists $s\in H^0(X,L)$ such that $s$ does not vanish at each generic point of $Z$, and $\widehat\deg_{\ovl L}(s)\geq \lambda_{\min}(\ovl L)-\epsilon$. Let $Y=\mr{div}(s).$
    \begin{align*}
        \displaystyle&\frac{\ovl L^{\dim X}}{\deg_L(X)}=\frac{1}{\deg_L(X)}\left(\ovl L^{\dim X-1}\cdot[Y]-\displaystyle\sum_{\nu\in\mathfrak B^{(1)}}\int \log\lVert s\rVert_{\phi_\nu}(x)c_1(\ovl L_\nu)^{\dim X}\right)
        \\
        &\geq \frac{\ovl L^{\dim X-1}\cdot[Y]}{\deg_L(X)}+\widehat{\deg}_{\ovl L}(s)\geq \frac{\ovl L^{\dim X-1}\cdot [Y]}{\deg_L(X)}+\lambda_{\min}(\ovl L)-\epsilon.
    \end{align*}
    Let $Y=\sum a_i Y_i$ where each $Y_i$ is a prime divisor. Then \begin{align*}
        & \ovl L^{\dim Y}\cdot[Y]=\sum a_i \ovl L^{\dim Y}\cdot[Y_i]\\ &\geq \sum {a_i}\deg_{L}(Y_i)(\mr{ess}_{\ovl L}(Y_i)+\dim Y\cdot\lambda_{\min}(\ovl L|_{Y_i}))\\
        &\geq \deg_L(Y)(\mr{ess}_{\ovl L}(X)+\dim Y\cdot \lambda_{\min}(\ovl L)-\epsilon),
    \end{align*}
    where the first inequality is due to induction hypothesis, and the second is due to Proposition \ref{prop_ample_fin_pullback} (2) and the following: $$\mr{ess}_{\ovl L}(Y_i)\geq \inf_{x\in (Y_i\setminus Z)(\ovl K)} h_{\ovl L}(x)\geq \inf_{x\not\in Z(\ovl K)} h_{\ovl L}(x)\geq \mr{ess}_{\ovl L}(X)-\epsilon.$$
    Notice that $\deg_L(X)=\deg_L(Y)$, it follows that $$\displaystyle\frac{\ovl L^{\dim X}}{\deg_L(X)}\geq \mr{ess}_{\ovl L}(X)+\dim X\cdot \lambda_{\min}(\ovl L)-2\epsilon.$$
    Since $\epsilon$ is arbitrary, we are done.
\end{proof}

\section{Appendix B}\label{sec_Appendix_B}
In this subsection, we sketch an alternative proof of the Bogomolov conjecture for semiabelian varieties over $\ovl\Q$, by reducing to the almost-split case proved in \cite{cham2000pet}. We refer the reader to \cite{zhang1995adelic} for the theory of adelic line bundles over number fields. 

We begin by recalling the canonical line bundle on $\mbb P^1_\Z$.
The tautological bundle $O(1)$ on $\mbb P^{1}(\C)$ is equipped with the \emph{canonical metric}:
$$\lVert X \rVert([x;y]):=\frac{\lvert x\rvert}{\max\{\lvert x\rvert,\lvert y\rvert\}}$$
where $[x;y]\in \mbb P^1(\C).$
We denote by $\ovl {O(1)}$ the adelic line bundle induced by the Hermitian line bundle $(O_{\mbb P^1_{\Z}}(1),\lVert\cdot\rVert).$ For $t'>0$, the canonical adelic line bundle $\ovl M'$ on $\mbb{G}_{m}^{t'}\subset (\mbb P^1)^{t'}$ is exactly $\underbrace{\ovl{O(1)}\boxtimes\cdots\boxtimes\ovl{O(1)}}_{t'\text{ times}}.$

Let $0\rightarrow \mbb G_m^t\rightarrow G\rightarrow A\rightarrow 0$ be a semiabelian variety over $\Q$. Let $\ovl N$ be the canonical line bundle associated with a symmetic ample line bundle $N$ on $A$. Let $\ovl M$ be the canonical line bundle on $\ovl G$ associated with the boundary divisor as in \cite{Kuhne2024semibogo}. Also, we take $\ovl L=\ovl M+\ovl{\pi}^*\ovl N$.

\begin{theo}\label{theo_torsion_test}
    Let $X\subset \mathbb G_m^{t'}\times G_{\ovl \Q}$ be a closed subvariety. 
    Let $p_2:(\mbb P^1)^{t'}\times \ovl G_{\ovl \Q}\rightarrow \ovl G_{\ovl \Q}$ and $p:(\mbb P^1)^{t'}\times \ovl G_{\ovl \Q}\rightarrow (\mbb P^1)^{t'}\times A_{\ovl \Q}.$
    Assume that there exists a subtorus $Y\subset \mathbb G_m^{t'}$ such that 
    $p|_{X}:X\rightarrow Y\times A$ is generically finite.
    If $p_2^*\ovl M\cdot p^*(\ovl M'\boxtimes \ovl N)^{\dim X}\cdot [\ovl X]=0,$
    then $G$ is almost-split.
\end{theo}
\begin{proof}
    The proof is identical with Theorem \ref{theo_constant_torsion_test}, so we only give a sketch.
    We may verify that Lemma \ref{lemm_aux} holds over number fields as well by using the fundamental inequality in \cite[Theorem 5.2]{zhangposvar}. 
    
    Note that by \cite[Theorem 7.1]{cham2000pet}, a closed subvariety in $\mbb G_m^{t'}\times A_{\Q}$ is of form $W\times B$ where $W$ is a torsion translate of a subtorus, and $B$ is a torsion translate of an abelian subvariety. Reproducing the proof of Theorem \ref{theo_constant_torsion_test} in the $\ovl K/k$-free trace case, that is, $A_1=A$, we concludes the proof.
\end{proof}

Then we can prove the Bogomolov conjecture for semiabelian varieties.
\begin{theo}[Bogomolov Conjecture for semi-abelian varieties]
    Let $X\subset G_{\ovl \Q}$ be a closed subvariety. Then $\{x\in X(\ovl \Q)\mid h_{\ovl L}(x)\leq \epsilon\}$ is Zariski dense in $X$ for every $\epsilon>0$ if and only if $X$ is a torsion translate of a subgroup.
\end{theo}
\begin{proof}
It suffices to prove the "only if" part.
After taking the quotient of $G_{\ovl \Q}$ by $\mr{Stab}(X)$, we may assume that $\mr{Stab}(X)=0$ and $\mr{ess}_{\ovl L}(X)=0$.

We may take sufficiently large $n$ such that there exists an irreducible component $X'$ of $X^n_{/S}$, which is generically finite to $\alpha_n(X')$ as in Proposition \ref{prop_rel_FZ_gen_fin}. Since $\alpha_n(X')$ is of essential minimum $0$, by \cite[Theorem 7.1]{cham2000pet}, $\alpha_n(X')\subset \mathbb G_m^{t(n-1)}\times A_{\ovl \Q}$ is a torsion translate of semiabelian subvariety. After taking an isogeny and shrinking $A$, we may assume that $\alpha_n(X')=Y\times A$ where $Y$ is a subtorus of $\mathbb G_m^{t(n-1)}$. Note that $\beta_n(X')$ also has essential minimum $0$, and is generically finite to $\alpha_n(X')$. Applying Theorem \ref{theo_torsion_test}, we obtain that $G$ is almost-split. Using \cite[Theorem 7.1]{cham2000pet} again, the result follows.
\end{proof}

\bibliographystyle{plain}
\bibliography{mybibliography}
\end{document}

%% file: autres_packages.tex
\usepackage[T1]{fontenc}
\usepackage[english,francais]{babel}
\usepackage{graphicx}
\usepackage{amssymb}
 \usepackage{url}
\usepackage[new]{old-arrows}
\usepackage{mathtools}
\usepackage{tikz}
\usepackage{tikz-cd} 
\usepackage{amsmath}
\usepackage[all]{xy}
\usetikzlibrary{matrix}
\usepackage{smfthm}
\usepackage{mathrsfs}
\usepackage{pb-diagram}
\usepackage{yfonts}
\usepackage[inline]{enumitem}
\usepackage{color}
\usepackage{verbatim}
\usepackage{dsfont}
\usepackage{appendix}
\usetikzlibrary{arrows, matrix} 

%% file: macros.tex
\def\R{\mathbb{R}}
\def\Q{\mathbb{Q}}
\def\C{\mathbb{C}}
\def\Z{\mathbb{Z}}
\def\N{\mathbb{N}}
\def\O{\mathcal{O}}
\def\L{\Lambda}

\def\deg{{\mathrm{deg}}}

\def\an{\mathrm{an}}

\def\ot{\otimes}
\def\mr{\mathrm}
\def\mc{\mathcal}
\def\mbb{\mathbb}

\def\ovl{\overline}

\def\scrA{\mathcal{A}}

\usepackage[normalem]{ulem}
\definecolor{mred}{rgb}{0.83, 0.0, 0.0}
\definecolor{darkspringgreen}{rgb}{0.09, 0.45, 0.27}
\definecolor{ruby}{rgb}{0.88, 0.07, 0.37}
\def\colorsout#1{\bgroup\markoverwith{\textcolor{#1}{\rule[0.5ex]{2pt}{0.7pt}}}\ULon} 
\def\coloruline#1{\bgroup\markoverwith{\textcolor{#1}{\rule[-0.5ex]{2pt}{0.7pt}}}\ULon} 

\makeatletter
\newcommand\tint{\mathop{\mathpalette\tb@int{t}}\!\int}
\newcommand\bint{\mathop{\mathpalette\tb@int{b}}\!\int}
\newcommand\tb@int[2]{%
  \sbox\z@{$\m@th#1\int$}%
  \if#2t%
    \rlap{\hbox to\wd\z@{%
      \hfil
      \vrule width .35em height \dimexpr\ht\z@+1.4pt\relax depth -\dimexpr\ht\z@+1pt\relax
      \kern.05em 
    }}
  \else
    \rlap{\hbox to\wd\z@{%
      \vrule width .35em height -\dimexpr\dp\z@+1pt\relax depth \dimexpr\dp\z@+1.4pt\relax
      \hfil
    }}
  \fi
}
\makeatother

\usepackage{etoolbox}
\makeatletter
\newcommand*\suppresschapternumber{%
  \let\@makechapterhead\@makeschapterhead
  \patchcmd{\@chapter}
    {\protect\numberline{\thechapter}}
    {}
    {}{}%
}
\newcommand*\removedotbetweenchapterandsection{%
  \renewcommand\thesection{\thechapter\@arabic\c@section}%
}
\makeatother